\documentclass[a4paper]{amsart}

\usepackage{amssymb}
\usepackage{amsthm}  
\usepackage{amsmath} 
\usepackage{amscd} 
\usepackage{float}
\usepackage[all]{xypic}
\usepackage{url}
\usepackage{graphicx}
\usepackage{hyperref}
\usepackage{color}
\hypersetup{linktocpage}

\newcommand{\om}{\omega}

\newcommand{\si}{\sigma}

\newcommand{\ba}{\mathcal{G}}

\newcommand{\fg}{\mathfrak g}

\newcommand{\fp}{\mathfrak p}

\newcommand{\fk}{\mathfrak k}

\newcommand{\id}{{\rm id}}

\newcommand{\C}{\mathbb{C}}

\theoremstyle{plain}

\theoremstyle{plain}
\newtheorem{thm*}{Theorem}[section]
\newtheorem{prop*}[thm*]{Proposition}
\newtheorem{lem*}[thm*]{Lemma}
\newtheorem{cor*}[thm*]{Corollary}

\theoremstyle{definition}
\newtheorem{def*}{Definition}

\theoremstyle{remark}
\newtheorem{rem*}{Remark}

\begin{document}
\title[Fundamental invariants of 2--nondegenerate CR geometries]{Fundamental invariants of 2--nondegenerate CR geometries with simple models}
\author{Jan Gregorovi\v c}
\address{Faculty of Science, University of Hradec Kr\'alov\'e, Rokitansk\'eho 62, Hradec Kr\'alov\'e 50003, Czech Republic and Faculty of Mathematics, University of Vienna, Oskar Morgenstern Platz 1, 1090 Wien, Austria}
 \email{jan.gregorovic@seznam.cz}
\subjclass[2010]{32V40, 32V35, 32V05, 53A55, 53C10}
 \thanks{The author gratefully acknowledges support via Czech Science Foundation (project no. 19-14466Y). The author also acknowledges partial support via FWF}
\maketitle
\begin{abstract}
This article studies the fundamental invariants of  2--nondegenerate CR geometries with simple models. We show that there are two sources of these invariants. The first source is the harmonic curvature of the parabolic geometry that appears (locally) on the leaf space of the Levi kernel. The second source is the difference between the complex structure on the complex tangent space of the CR geometry and the complex structure on the correspondence space to the underlying parabolic geometry. We show that the later fundamental invariants appear only when the model is generic and if they vanish, then the solution of the local equivalence problem of 2--nondegenerate CR geometries with simple models is provided by the Cartan connection of the underlying parabolic geometry. We show that nontrivial examples of CR geometries with the later fundamental invariants can be obtained as deformations of the models.
\end{abstract}

\tableofcontents
\section{Introduction}

A way toward the solution of the problem of the biholomorphic equivalence for real submanifolds in the complex space is to compare the induced CR geometries with appropriate model CR geometries. In the case of the Levi--nondegenerate real hypersurfaces in $\C^N$, the maximally symmetric models are quadrics that can be identified (depending on the signature) with (the CR geometries on) the homogeneous spaces $G/H=SU(p+1,N-p)/P$, where $SU(p+1,N-p)$ is the (real simple) CR automorphism group of the quadric and the stabilizer $P$ is a particular parabolic subgroup of $SU(p+1,N-p)$. The difference between a hypersurface and the model $SU(p+1,N-p)/P$ is measured by invariants, i.e., quantities and tensors that depend only on the hypersurface and the chosen point. Among all the invariants there is a distinguished class of \emph{fundamental invariants} that provide all the necessary invariants for the solution of the equivalence problem using invariant differentiation. The fundamental invariants for $N=2$ were found by Cartan \cite{cartan} and for $N> 2$, by Tanaka \cite{Ta62,Ta70,Ta76} and Chern and Moser \cite{chern}. These invariants have a uniform description as \emph{harmonic curvatures} in the theory of parabolic geometries, see \cite[Section 4.2]{parabook}. The theory of harmonic curvatures of parabolic geometries provides all fundamental invariants of \emph{Levi--nondegenerate} CR submanifolds with semisimple models, c.f. \cite{MS,parabook,SS00,SS12}.

In \cite{Gr19}, we solved an equivalence problem for a class of \emph{Levi--degenerate} real submanifolds in $\C^N$, i.e., CR geometries with non--trivial Levi kernel, that have \emph{maximally symmetric models} with \emph{real simple} CR automorphism groups. We proved in \cite[Theorem 1.1]{Gr19} that the structure of the maximally symmetric models is related to the theory of parabolic geometries. We recall the details about (semi)simple Lie group, their parabolic subgroups, and parabolic geometries in Section \ref{Kap1}. In particular, if the homogeneous space $G/H$ is the maximally symmetric model for the real simple CR automorphism group $G$, then there is a parabolic subgroup $P$ of $G$ such that the leaves of the Levi kernel can be identified with a (noneffective) pseudo--Hermitian symmetric space $P/H$, i.e., $G/P$ is the leaf space of the Levi kernel. Moreover, if we denote by $$(\fg_{-k}\oplus \dots \oplus \fg_{-1})\oplus (\fg_{0}\oplus \dots \oplus \fg_{k})=\fg_-\oplus \fp$$ the grading of the Lie algebra $\fg$ of $G$ that corresponds to $P$, then the Levi kernel can be simultaneously identified:
\begin{enumerate}
\item with the $-1$--eigenspace $\fk$ of symmetries at the origin of the pseudo--Hermitian symmetric space $P/H$,
\item with subset $\fk$ of elements of $\fg_0$ that are \emph{complex antilinear} w.r.t. a complex structure $I$ on $\fg_{-1}$, i.e., the CR geometry is \emph{2--nondegenerate} with \emph{second--order Levi--Tanaka algebra} $(\fg_-,I,\fk)$.
\end{enumerate}
We recall the details about these identifications and the method of how to assign a second--order Levi--Tanaka algebra $(\fg_-,I,\fk)$ to each point of  a Levi--degenerate real submanifolds in $\C^N$ in Section \ref{Kap2}. The second--order Levi--Tanaka algebra $(\fg_-,I,\fk)$ generalizes the usual \emph{Levi--Tanaka algebra} $(\fg_-\oplus \fk,I)$ of a CR geometry by adding data about the complex antilinear action of the Levi kernel $\fk$ on the quotient $\fg_{-1}$ of complex tangent space by the Levi kernel.

In this article, we look on the fundamental invariants of  2--nondegenerate CR geometries $(M,\mathcal{D}, \mathcal{I})$ that have at each point \emph{the same} second--order Levi--Tanaka algebra $(\fg_-,I,\fk)$ corresponding to some maximally symmetric model $G/H$ with simple real $G$. 

We show in Proposition \ref{under} that the relation to parabolic geometries goes beyond models and that we can always find an underlying parabolic geometry of type $(G,P)$. This underlying geometry is \emph{uniquely} determined by the CR geometry with exception of the class of CR geometries that we discuss in Section \ref{KapSp}, where we need to choose the flat underlying parabolic geometry for our further consideration. 

We show in Theorems \ref{fin1} and \ref{fin2} that the underlying parabolic geometry provides a (different) 2--nondegenerate CR geometry $(M,\mathcal{D}, \mathcal{\tilde I})$ on $M$ with second--order Levi--Tanaka algebra $(\fg_-,I,\fk)$. The comparison of $(M,\mathcal{D}, \mathcal{\tilde I})$ with $(M,\mathcal{D},\mathcal{I})$ allows us to characterize all fundamental invariants of 2--nondegenerate CR geometries with simple models. In particular, all fundamental invariants can be derived from the difference $\mathcal{I}-\mathcal{\tilde I}$ and the fundamental invariants of the underlying parabolic geometries the so--called harmonic curvatures. It turns out that several different situations arise depending on the \emph{genericity} of the second--order Levi--Tanaka algebra $(\fg_-,I,\fk)$, i.e., depending on whether the assumption that the  Levi--Tanaka algebra is the same at each point is trivially satisfied on an open neighborhood of each point with  Levi--Tanaka algebra $(\fg_-,I,\fk)$ or not. In particular, if  $(\fg_-,I,\fk)$ is \emph{not} generic then we show in Theorems \ref{fin1} and \ref{fin2} that $\mathcal{\tilde I}=\mathcal{I}$ and the fundamental invariants are only the harmonic curvatures.  We show in Theorem \ref{harm} that not every harmonic curvature can appear as a fundamental invariant of a 2--nondegenerate CR geometry, because some harmonic curvatures provide obstruction for the (formal) integrability of the CR geometry, i.e., they appear only on almost CR geometries that can not be embedded into $\C^N$. We obtain the following classification of fundamental invariants:

\begin{enumerate}
\item The case of five--dimensional 2--nondegenerate CR submanifolds in $\C^3$ has two fundamental invariants. This is the only well--know case, c.f. \cite{Eb06,IZ13,Poc13,KK19}, and we discuss it in the Appendix.
\item The case of seven--dimensional 2--nondegenerate CR submanifold in $\C^{5}$ has two fundamental invariants. We show in Section \ref{genc} (together with explicit example) that one of them is given by the difference $\mathcal{\tilde I}-\mathcal{I}$ and if this invariant vanishes, then the other invariant is the harmonic curvature of the underlying $(2,3,5)$ geometry, i.e., the model has the exceptional Lie group $G_{2(2)}$ as its CR automorphism group.
\item The case of $N^2+N-1$--dimensional $2$--nondegenerate CR submanifolds in $\C^{\frac{N(N+1)}{2}}$ with $(N-1)N$--dimensional Levi--kernel, for $N>2$, has several fundamental invariants induced by the derivations of the difference $\mathcal{\tilde I}-\mathcal{I}$, see Theorem \ref{fin2} and example in Section \ref{KapSp}. If they vanish, then the CR geometry corresponds (locally) to an open subsets in $G/H$. 
\item There is a class of cases described in Theorem \ref{harm} that have the harmonic curvatures of the underlying parabolic geometry as the only fundamental invariants. Let us remark that in the case $\mathfrak{su}(p,4-p)$ we obtained stronger results than in \cite{Por15}, where the author obtains (due to the choice of normalization) the fundamental invariants but not the Cartan connection.
\item The remaining cases admit no fundamental invariants, i.e., the CR geometries correspond (locally) to open subsets in $G/H$. 
\end{enumerate}

We show that in the cases (1) and (3) from the above list, we can in the principle find (locally) \emph{all} 2--nondegenerate CR geometries of that type by deforming the complex structure $\mathcal{\tilde I}$ coming from the maximally symmetric model. However, in the case (3) the PDE's that are imposed by the (formal) integrability on the admissible deformations are too complicated to be solved (by PDE solver in Maple software), see also the example in Section \ref{KapSp}. On the other hand, in the case (1), we provide in Appendix all admissible deformations of the flat model.

\section{Parabolic geometries and gradings of Lie algebras}\label{Kap1}

In this section,  we review the structure theory of the (regular) parabolic geometries from \cite{parabook}.

 Let us recall that each pair of a (semi)simple Lie group $G$ and its parabolic subgroup $P$ is related to a $|k|$--grading of the Lie algebra $\fg$ of $G$, i.e., to a direct sum decomposition $\fg=\fg_{-k}\oplus \dots \oplus \fg_{k}$ such that $[\fg_{i},\fg_{j}]\subset \fg_{i+j}$ with the following properties:

\begin{enumerate}
\item $\fp=\fg_{0}\oplus \dots \oplus \fg_{k}$ is the Lie algebra of $P$ preserving the filtration $$\fg^{i}:=\fg_{i}\oplus \dots \oplus \fg_{k}$$ for $i\leq k$, i.e., $[\fp,\fg^{i}]\subset \fg^{i}.$
\item There is a set of simple roots $\Sigma$ such that root spaces $\fg_\alpha \subset \fg_1$, $\fg_{-\alpha} \subset \fg_{-1}$ and $\fg_\beta\subset \fg_0$ for simple roots $\alpha\in \Sigma, \beta\notin \Sigma$ generate the whole grading. In particular, the $\fg_0$--module $\fg_{-1}$ generates the whole negative part $\fg_-$ of the grading and the length $k$ of the grading can be determined from the highest root.
\item There is a subgroup $G_0$ of $P$ of grading preserving elements with Lie algebra $\fg_0$ and $\exp(\fp_+)$ is the nilradical of $P$ with Lie algebra $\fp_+=\fg_{1}\oplus \dots \oplus \fg_{k}$. Moreover, $P=G_0\rtimes \exp(\fp_+)$.
\end{enumerate}

The parabolic geometries form a class of geometric structures on smooth manifolds that do not have much in common on a first glance. However, the problem of equivalence of (regular) parabolic geometries has a common solution using a (normal) Cartan connection $\tilde \om$ of type $(G,P)$, where $G$ is semisimple Lie group and $P$ is its parabolic subgroup. Let us recall that a Cartan geometry $(\ba\to N,\tilde \om)$ on a smooth manifold $N$ of type $(G,P)$ is

\begin{enumerate}
\item a principal $P$--bundle $\ba$ over $N$,
\item a $\fg$--valued $P$--equivariant absolute parallelism $\tilde \om$ on $\ba$ that reproduces the fundamental vector fields of the $P$--action and provides isomorphism $TN\cong \ba\times_P \fg/\fp.$
\end{enumerate}

The Cartan connection is called normal if the curvature $$\kappa(X,Y):=[X,Y]+\om([\om^{-1}(X),\om^{-1}(X)])$$ for $X,Y\in \fg/\fp$ satisfies $\partial^*\kappa=0$ for the Kostant's codifferential $$\partial^*: \wedge^2(\fg/\fp)^*\otimes \fg\to (\fg/\fp)^*\otimes \fg$$ that is defined for dual bases (w.r.t. Killing form of $\fg$) of $X_i$ of $\fg/\fp$ and $Z_i$ of $\fp_+$ as $$\partial^* (\Omega)(X):= \sum_i 2[Z_i,\Omega(X,X_i)]-\Omega([Z_i,X],X_i).$$

The parabolic geometry is called regular if the curvature $\kappa$ has only components of positive homogeneity w.r.t. the grading, i.e., for $X\in \fg^{i}$ and $Y\in \fg^{j}$, $i,j<0$ is $\kappa(X,Y)\in\fg^{i+j+1}.$ The fundamental invariants of regular parabolic geometries are the harmonic curvatures, i.e., the components of $\kappa$ in $Ker(\partial^*)/Im(\partial^*)$.

The quotients $\fg^{i}/\fp$ define distinguished distributions on $N$ and $G_0$--invariant objects define distinguished geometric structures on $$T^{-1}N:=\ba\times_P \fg^{-1}/\fp.$$ The case of Levi--nondegenrate CR hypersurfaces in $\C^N$ is typical example of regular parabolic geometry (of type $(SU(p+1,N-p),P_{1,N})$ for the set $\Sigma=\{\alpha_1,\alpha_N\}$ of simple roots (with ordering from \cite{parabook})). In this case, $T^{-1}N=\mathcal{D}$ is the complex tangent space and $G_0=CSU(p,N-p-1)$ preserves the complex structure $\mathcal{I}$ on the complex tangent space and the Levi bracket. The negative part of the grading $\fg_-$ corresponding to a contact distribution is in this case the Heisenberg algebra and provides the Levi--Tanaka algebra $(\fg_-,I)$.

There are several ways how to describe the underlying geometric structures of the parabolic geometries in detail. We recall in Section \ref{Kapunder} a description that is applicable also for the 2--nondegenerate CR geometries with simple models.

\section{2--nondegenerate CR geometries and bigradings of Lie algebras}\label{Kap2}

In this section, we recall the relevant results on 2--nondegenerate CR geometries with simple models from \cite{Gr19}.

For a submanifold $M$ in $\C^N$, we consider the corresponding CR geometry $(M,\mathcal{D},\mathcal{I})$, where $$\mathcal{D}:=TM\cap i(TM)$$ is the complex tangent space and $\mathcal{I}$ is the complex structure on $\mathcal{D}$ induced by multiplication by $i$ on $T\C^N$.  We denote by $\mathcal{K}$ the Levi kernel, i.e., the maximal complex subspace of $\mathcal{D}$ such that $[\mathcal{K},\mathcal{D}]\subset \mathcal{D}$. Then we proceed with a choice of a (local) frame of $\mathcal{D}/\mathcal{K}$ and assume that it identifies (pointwise) $\mathcal{D}/\mathcal{K}$ with $\fg_{-1}$. Unlike in the construction of Levi--Tanaka algebra, we do not pick frame of $\mathcal{D}$, because the 2--nondegeneracy can be characterized according to \cite{Fr74,Fr77} by existence of (local) frame of $\mathcal{D}/\mathcal{K}$ satisfying the condition (2) of the following definition:

\begin{def*}\label{2olta}
We can say that $(M,\mathcal{D},\mathcal{I})$ is a 2--nondegenerate CR geometry with second--order Levi--Tanaka algebra $(\fg_-,I,\fk)$ if there is a (local) frame of $\mathcal{D}/\mathcal{K}$ such that:

\begin{enumerate}
\item the graded Lie algebra defined (via the Lie bracket of vector fields) iteratively as $$\fg_{-i}:=[\fg_{-i+1},\fg_{-1}] \mod \fg_{-i+1}$$ is at each point isomorphic to $\fg_-$,
\item the complex antilinear part of the bracket $$[\mathcal{K},\mathcal{D}/\mathcal{K}]\subset \mathcal{D}/\mathcal{K}$$ defines at each point $x\in M$ bijection of $\mathcal{K}_x$ with a subspace $\fk\subset \frak{gl}(\fg_{-1})$ consisting of complex antilinear endomorphisms of $\fg_{-1}$ for complex structure $I$ on $\fg_{-1}$ induced by $\mathcal{I}$.
\end{enumerate}
\end{def*}

Let us remark that $(\fg_-\oplus \fk,I)$ is the usual Levi--Tanaka algebra of such a CR geometry and that the inclusion $\fk\subset \frak{gl}(\fg_{-1})$ contains the additional information about the Levi kernel (that is of the second--order w.r.t. to the Levi bracket).

Let us recall how to find the (local) frame of $\mathcal{D}/\mathcal{K}$ that has properties from Definition \ref{2olta} and how to decide whether the CR geometry is 2--nondegenerate with second--order Levi--Tanaka algebra $(\fg_-,I,\fk)$. Since the Lie bracket of sections of $\mathcal{D}$ is algebraic modulo $\mathcal{D}$, the Levi kernel $\mathcal{K}$ is characterized by linear equations in an arbitrary chosen (local) frame of $\mathcal{D}$. Therefore, we can easily choose a frame of $\mathcal{D}$ decomposing to the frame of $\mathcal{D}/\mathcal{K}$ and frame of $\mathcal{K}$. Therefore, iterating the brackets of sections of $\mathcal{D}$, we obtain the algebraic brackets $[\fg_{-i+1},\fg_{-1}] \mod \fg_{-i+1}$ that are independent (up to isomorphisms) on the choice of the frame of $\mathcal{D}$. Of course, if the Levi--Tanaka algebra $(\fg_-\oplus \fk,I)$ is not generic, then the algebraic brackets depend on the point of $M$ and we need to assume that they are all the same. Then we can choose the frame of $\mathcal{D}/\mathcal{K}$ such that we obtain the same Lie algebra $\fg_-$ at all points of $M$. Next, we check the 2--nondegeneracy condition (2) in the Definition \ref{2olta} and obtain (pointwise) inclusion of $\mathcal{K}$ into $\frak{gl}(\fg_{-1})$. Of course, if the second--order Levi--Tanaka algebra $(\fg_-\oplus \fk,I)$ is not generic, then the inclusions depend on the point of $M$ and we need to assume that they are all the same. Then we can change the frame of $\mathcal{D}/\mathcal{K}$ without changing $\fg_-$ (i.e. by a function with values in grading preserving automorphisms of $\fg_-$) such that we obtain the same subspace $\fk$ of $\frak{gl}(\fg_{-1})$ at all points of $M$.

We have proved in \cite{Gr19} that the second--order Levi--Tanaka algebras $(\fg_-,I,\fk)$ corresponding to CR geometries with simple models $G/H$ are related to specific bigradings of complex simple Lie groups. In particular, the second--order Levi--Tanaka algebra is uniquely determined by:
\begin{enumerate}
\item a simple real Lie algebra $\fg$,
\item a set of simple roots $\Sigma_1$ providing grading $\fg_{-k}\oplus \dots \oplus \fg_k=\fg_-\oplus \fp$ of $\fg$, which provides $\fg_-$,
\item a Hermitian symmetric pair $(\fg_0,\fg_{0,I})$ such that $\fg_{0,I}$ acts on $\fg_{-1}$ by a complex representation, $I$ is defined as $\pm i$ on irreducible components of $\fg_{-1}$, where the signs are assigned in a way that $$[I(X),I(Y)]=[X,Y]$$ for all $X,Y\in \fg_{-1}$, which provides $I$ and $\fk$,
\item a set of simple roots $\Sigma_2$ providing grading of the complexification of $\fg$ such that the corresponding bigrading $\fg_{a,b}$ of the complexification of $\fg$ given by $\Sigma_1,\Sigma_2$ satisfies
\begin{align*}
\fg_{0,I}\otimes \C&=\fg_{0,0}\\
\fg_{0}\otimes \C&=\fg_{0,-1}\oplus \fg_{0,0}\oplus \fg_{0,1}\\
\fg_{-1}\otimes \C&=\fg_{-1,-1}\oplus \fg_{-1,0}\\
\fg_{-2}\otimes \C&=\fg_{-2,-1}.
\end{align*}
\end{enumerate}

Let us emphasize that according to \cite[Theorem 1.3]{Gr19} the Hermitian symmetric pair from above point (3) provides a set of simple roots $\Sigma_2$ from above point (4), but the converse is not always true, because there can be more real forms $\fg_{0,I}$ of $\fg_{0,0}$ defining a Hermitian symmetric pair $(\fg_0,\fg_{0,I})$. Therefore, the Tables \ref{realclasA} and \ref{realclasB} obtained in \cite[Theorem 1.3]{Gr19} contain the information about $\fg, \Sigma_1,\Sigma_2$ and about the connected components $\fg_{0,I}'$ of $\fg_{0,I}$ that are not connected components of $\fg_0$.

\begin{table}[H]\caption{Classification of classical simple maximally symmetric models of 2--nondegenerate CR geometries}\label{realclasA}
\begin{tabular}{|c|c|c|}
\hline
$\mathfrak{g}$ & $\Sigma_1$ & restrictions\\
$\fg_{0,I}'$& $\Sigma_2$ &  \\
\hline
$\mathfrak{sl}(n+1,\mathbb{R})$&$\{\alpha_r,\alpha_{r+2s}\}$ & \\
$\frak{gl}(s,\C)$ &$\{\alpha_{r+s}\}$& \\
\hline
$\mathfrak{sl}(n+1,\mathbb{H})$&$\{\alpha_{2r},\alpha_{2s}\}$   & \\
$\frak{gl}(s-r,\C)$&$\{\alpha_{r+s}\}$& \\
\hline
$\mathfrak{su}(p,n+1-p)$&$\{\alpha_r,\alpha_{n-r}\}$ & $r<s<n-r$ \\
$\frak{su}(q,s-r-q)\oplus \frak{u}(p-r-q,n+1-p-s+q)$&$\{\alpha_s\}$ &$0\leq q<s-r$ \\
\hline
$\mathfrak{so}(p,2n+1-p)$& $\{\alpha_r,\alpha_{r+2}\}$ & \\
 $\frak{so}(2)$&$\{\alpha_{r+1}\}$ &\\
 \hline
$\mathfrak{so}(p,2n-p)$ or $\mathfrak{so}^*(2n)$& $\{\alpha_r,\alpha_{r+2}\}$ & $r<n-3$\\
 $\frak{so}(2)$&$\{\alpha_{r+1}\}$ &\\
\hline
$\mathfrak{so}(p,q)$ or $\mathfrak{so}^*(2n)$ & $\{\alpha_2\}$ & \\
$\frak{so}(2)$& $\{\alpha_1\}$&\\
\hline
$\mathfrak{so}(p,2n-p)$& $\{\alpha_{p-2q}\}$  &$1<p-2q<n-1$\\
$\frak{u}(q,n-p+q)$& $\{\alpha_n\}$&$0\leq q$\\
\hline
$\mathfrak{so}^*(2n)$& $\{\alpha_{r}\}$&$r<n-1$\\
$\frak{u}(p,n-r-p)$& $\{\alpha_n\}$&\\
\hline
$\mathfrak{so}(n,n), \mathfrak{so}(n-1,n+1)$ or $\mathfrak{so}^*(4m+2)$& $\{\alpha_{n-3},\alpha_{n-1},\alpha_n\}$ &$n=2m+1$ \\
$\frak{so}(2)$& $\{\alpha_{n-2}\}$&\\
\hline
$\mathfrak{sp}(2n,\mathbb{R})$& $\{\alpha_r\}$&$0\leq p$ \\
$\frak{u}(p,n-r-p)$ &$\{\alpha_n\}$&\\
\hline
$\mathfrak{sp}(p,n-p)$& $\{\alpha_r\}$&  \\
$\frak{u}(p-r,n-p-r)$&$\{\alpha_n\}$&\\
\hline
\end{tabular}
\end{table}
\noindent
\begin{table}[H]\caption{Classification of exceptional simple maximally symmetric models of 2--nondegenerate CR geometries}\label{realclasB}
\begin{tabular}{|c|c|c|c|}
\hline
$\mathfrak{g}$ & $\Sigma_1$ & $\Sigma_2$ &$\fg_{0,I}'$ \\
\hline
$\mathfrak{g}_2(2)$ & $\{\alpha_1\}$&$\{\alpha_2\}$&$\frak{so}(2)$ \\
\hline
$\mathfrak{f}_4(4)$ & $\{\alpha_2\}$&$\{\alpha_1\}$&$\frak{so}(2)$ \\
$\mathfrak{f}_4(4)$ & $\{\alpha_1,\alpha_3\}$&$\{\alpha_2\}$&$\frak{so}(2)$ \\
\hline
$\mathfrak{e}_6(6)$ & $\{\alpha_2\}$&$\{\alpha_1\}$&$\frak{so}(2)$ \\
$\mathfrak{e}_6(2)$ & $\{\alpha_6\}$&$\{\alpha_1\}$&$\frak{so}(2)\oplus \frak{u}(2,3)$\\
$\mathfrak{e}_6(-14)$ & $\{\alpha_6\}$&$\{\alpha_1\}$&$\frak{so}(2)\oplus \frak{u}(5)$ \\
$\mathfrak{e}_6(6)$ & $\{\alpha_3\}$&$\{\alpha_6\}$&$\frak{so}(2)$ \\
$\mathfrak{e}_6(2)$ & $\{\alpha_3\}$&$\{\alpha_6\}$&$\frak{so}(2)$\\
$\mathfrak{e}_6(6)$ & $\{\alpha_1,\alpha_3\}$&$\{\alpha_2\}$&$\frak{so}(2)$ \\
$\mathfrak{e}_6(6)$ & $\{\alpha_2,\alpha_4,\alpha_6\}$&$\{\alpha_3\}$&$\frak{so}(2)$ \\
$\mathfrak{e}_6(2)$ & $\{\alpha_2,\alpha_4,\alpha_6\}$&$\{\alpha_3\}$&$\frak{so}(2)$ \\
\hline
$\mathfrak{e}_7(7)$ & $\{\alpha_2\}$&$\{\alpha_1\}$&$\frak{so}(2)$\\
$\mathfrak{e}_7(-5)$ & $\{\alpha_2\}$&$\{\alpha_1\}$&$\frak{so}(2)$ \\
$\mathfrak{e}_7(-25)$ & $\{\alpha_2\}$&$\{\alpha_1\}$&$\frak{so}(2)$ \\
$\mathfrak{e}_7(7)$ & $\{\alpha_5\}$&$\{\alpha_6\}$&$\frak{so}(2)$ \\
$\mathfrak{e}_7(-5)$ & $\{\alpha_5\}$&$\{\alpha_6\}$&$\frak{so}(2)$ \\
$\mathfrak{e}_7(7)$ & $\{\alpha_6\}$&$\{\alpha_1\}$&$\frak{so}(2)\oplus \frak{so}(4,4)$ \\
$\mathfrak{e}_7(-25)$ & $\{\alpha_6\}$&$\{\alpha_1\}$&$\frak{so}(2)\oplus \frak{so}(1,7)$ \\
$\mathfrak{e}_7(7)$ & $\{\alpha_1,\alpha_3\}$&$\{\alpha_2\}$&$\frak{so}(2)$\\
$\mathfrak{e}_7(7)$ & $\{\alpha_2,\alpha_4\}$&$\{\alpha_3\}$&$\frak{so}(2)$ \\
$\mathfrak{e}_7(-5)$ & $\{\alpha_2,\alpha_4\}$&$\{\alpha_3\}$&$\frak{so}(2)$ \\
$\mathfrak{e}_7(7)$ & $\{\alpha_4,\alpha_6\}$&$\{\alpha_5\}$&$\frak{so}(2)$ \\
$\mathfrak{e}_7(-5)$ & $\{\alpha_4,\alpha_6\}$&$\{\alpha_5\}$&$\frak{so}(2)$ \\
$\mathfrak{e}_7(7)$ & $\{\alpha_3,\alpha_5,\alpha_7\}$&$\{\alpha_4\}$&$\frak{so}(2)$ \\
\hline
$\mathfrak{e}_8(8)$ & $\{\alpha_2\}$&$\{\alpha_1\}$&$\frak{so}(2)$\\
$\mathfrak{e}_8(-24)$ & $\{\alpha_2\}$&$\{\alpha_1\}$&$\frak{so}(2)$ \\
$\mathfrak{e}_8(8)$ & $\{\alpha_6\}$&$\{\alpha_7\}$&$\frak{so}(2)$ \\
$\mathfrak{e}_8(8)$ & $\{\alpha_5\}$&$\{\alpha_8\}$&$\frak{so}(2)$ \\
$\mathfrak{e}_8(8)$ & $\{\alpha_1,\alpha_3\}$&$\{\alpha_2\}$&$\frak{so}(2)$ \\
$\mathfrak{e}_8(-24)$ & $\{\alpha_1,\alpha_3\}$&$\{\alpha_2\}$&$\frak{so}(2)$ \\
$\mathfrak{e}_8(8)$ & $\{\alpha_2,\alpha_4\}$&$\{\alpha_3\}$&$\frak{so}(2)$ \\
$\mathfrak{e}_8(8)$ & $\{\alpha_3,\alpha_5\}$&$\{\alpha_4\}$&$\frak{so}(2)$ \\
$\mathfrak{e}_8(8)$ & $\{\alpha_5,\alpha_7\}$&$\{\alpha_6\}$&$\frak{so}(2)$ \\
$\mathfrak{e}_8(8)$ & $\{\alpha_4,\alpha_6,\alpha_8\}$&$\{\alpha_5\}$&$\frak{so}(2)$ \\
\hline
\end{tabular}
\end{table}

Let us recall the solution of equivalence problem for the CR geometries with second--order Levi--Tanaka algebra corresponding to our classification we obtained in \cite[Theorem 1.4]{Gr19}.

\begin{thm*}\label{abspar}
Suppose  $G/H=G/G_{0,I}\rtimes \exp(\fp_+)$ is a homogeneous model of 2--nondegenerate CR submanifold with second--order Levi--Tanaka algebra $(\fg_-,I,\fk)$ corresponding to one of the entries in Tables \ref{realclasA} and \ref{realclasB}. Then there are $G_{0,I}$--invariant normalization conditions that provide equivalence of categories between

\begin{itemize}
\item the category of 2--nondegenerate CR geometries $(M,\mathcal{D},\mathcal{I})$ with second--order Levi--Tanaka algebra $(\fg_-,I,\fk)$,
\item the category of $H$--fiber bundles $\ba\to M$ with a $G_{0,I}$--invariant $\fg$--valued absolute parallelism $\om$ satisfying the normalization conditions.
\end{itemize}

In particular, $dim(\fg)$ bounds the dimension of Lie algebra of infinitesimal CR automorphisms of all 2--nondegenerate CR geometries with second--order Levi--Tanaka algebra $(\fg_-,I,\fk)$ and $G/H$ is the maximally symmetric model.
\end{thm*}

We recall further details on the above solution of the equivalence problem and the normalization conditions in the next section.

\section{Weyl structures and underlying geometric structures}\label{Kapunder}

There is a common construction of a parabolic geometry $(\ba\to N,\tilde \om)$ of type $(G,P)$ and the absolute parallelism $\om$ on the $H$--fiber bundle $\ba\to M$ starting with an appropriate underlying geometric structure. This involves the so--called Weyl structures, i.e., sections of $$\si: \ba_0:=\ba/\exp(\fp_+)\to \ba$$ that are equivariant for the structure group $G_0$ on $\ba_0$ in the case of parabolic geometries or the structure group $G_{0,I}$ on $\ba_0$ in the case of 2--nondegenerate CR geometries with simple models. If we decompose the pullbacks $\si^*\tilde \om$ and $\si^*\om$ according to the grading (or bigrading if we complexify  $\si^*\om$), then the parts that do not depend on the choice of the Weyl structure provide a distinguished underlying geometric structure that we use as the starting point in the construction of $\tilde \om$ and $\om$. In this article, we will not need the explicit formulas how $\si^*\tilde \om$ and $\si^*\om$ depend on the choice of the Weyl structure, which can be found in \cite[Sector 5.1]{parabook} and \cite[Section 4.2]{Gr19}.

In the case of parabolic geometries, we start with the so--called pseudo--$G_0$--structures of type $\fg_-$, see \cite[Section 3.1]{parabook}:

\begin{def*}\label{infG0}
A pseudo--$G_0$--structure of type $\fg_-$ (a regular infinitesimal flag structure of type $(G,P)$ in terminology of  \cite{parabook}) on a manifold $N$ is a principal $G_0$--bundle $\ba_0$ over $N$ together with filtration $$T^{-k}\ba_0\supset \dots \supset T^{-1}\ba_0\supset T^{0}\ba_0$$ and collection of $G_0$--equivariant pseudo one--forms  $$\theta_i: T^{i}\ba_0\to \fg_i$$  for $i<0$ with the following properties:

\begin{enumerate}
\item $T^{0}\ba_0$ is the vertical bundle,
\item $Ker(\theta_i)=T^{i+1}\ba_0$,
\item $\theta_i|_{T^{i}\ba_0/T^{i+1}\ba_0}$ is isomorphism,
\item the Lie bracket of sections $\nu$ of $T^{i}\ba_0$ and $\xi$ of $T^{j}\ba_0$ for $i,j< 0$ is a section of $T^{i+j}\ba_0$ and for $i+j\geq -k$ holds $$\theta_{i+j}([\nu,\xi])=[\theta_i(\nu),\theta_j(\xi)],$$
\item there is an extension of $\theta_i$ to $\fg_{-}$--valued one--forms on $T\ba_0$ such that (4) holds for the extensions modulo $\fg_{i+j+1}$.
\end{enumerate}
\end{def*}

In the situations corresponding to our classification, we can distinguish the following three situations with respect to the underlying geometric structure required by the construction of $\tilde \om$ using the normalization condition $$\partial^*\kappa=0$$ for the curvature $\kappa$ of $\tilde \om$:

\begin{enumerate}
\item the parabolic geometries of type $(Sp(2N,\mathbb{R}),P_1)$ (the so--called contact projective geometries) that are determined by a choice of a class of (partial) Weyl connections on $T^{-1}N$, which correspond to complements of vertical bundle in $T^{-1}\ba_0$ given by $\si^*\tilde \om^{-1}(\fg_{-1})$ for all Weyl structures $\si$,
\item the parabolic geometries such that $k=2,\dim(\fg_{-2})=1$ or with type $(Sl(N+1,\mathbb{R}),P_{1,i})$ that are determined by the pseudo--$G_0$--structure of type $\fg_-$,
\item  the remaining parabolic geometries that are determined by the distribution $T^{-1}N$, i.e., the pseudo--$G_0$--structure $\ba$ of type $\fg_-$ is just the bundle of frames of $T^{-1}N$ with the canonical soldering pseudo one--forms.
\end{enumerate}

Let us recall that the Tanaka prolongation described in \cite{Ta70} provides an alternative construction of $\tilde \om$ starting with the pseudo--$G_0$--structures of type $\fg_-$ (with an additional reduction of the first prolongation in the case of contact projective structures) and there are also other prolongation procedures for parabolic geometries summarized in \cite[Appendix A]{parabook}.

In the case of 2--nondegenerate CR geometries with simple models, we start with the so--called infinitesimal pseudo--$\fg_0$--structure of type $\fg_-$, see \cite[Section 4.1]{Gr19}.

\begin{def*}
An infinitesimal pseudo--$\fg_0$--structure of type $\fg_-$ on the principal $G_{0,I}$--bundle $\ba_0$ of complex frames of $\mathcal{D}/\mathcal{K}$ is a collection of pseudo one--forms $$\theta_i: T^{i}\ba_0\to \fg_i$$  for $i<0$ on a filtration $$T^{-k}\ba_0\supset \dots \supset T^{-1}\ba_0\supset T^{0}\ba_0$$ with the following properties:

\begin{enumerate}
\item $T^{0}\ba_0$ is integrable bundle with transitive infinitesimal $\fg_0$--action, which can be extended to $T\ba_0$ in a way that all $\theta_i$ are $\fg_0$--equivariant,
\item $Ker(\theta_i)=T^{i+1}\ba_0$,
\item $\theta_i|_{T^{i}\ba_0/T^{i+1}\ba_0}$ is isomorphism,
\item the Lie bracket of sections $\nu$ of $T^{i}\ba_0$ and $\xi$ of $T^{j}\ba_0$ for $i,j<0$ is a section of $T^{i+j}\ba_0$ and for $i+j\geq -k$ holds $$\theta_{i+j}([\nu,\xi])=[\theta_i(\nu),\theta_j(\xi)],$$
\item there is an extension of $\theta_i$ to $\fg_{-}$--valued one--forms on $T\ba_0$ such that (4) holds for the extensions modulo $\fg_{i+j+1}$.
\end{enumerate}
\end{def*}

Let us emphasize that the transitive infinitesimal $\fg_0$--action on the principal $G_{0,I}$--bundle $\ba_0$ of complex frames of $\mathcal{D}/\mathcal{K}$ does not appear automatically and we proved its existence and uniqueness in \cite[Proposition 4.2]{Gr19} for the 2--nondegenerate CR geometries $(M,\mathcal{D},\mathcal{I})$ with second--order Levi--Tanaka algebra $(\fg_-,I,\fk)$ corresponding to one of the cases from Tables  \ref{realclasA} and \ref{realclasB}. In particular, according to \cite[Proposition 4.2]{Gr19}, the infinitesimal pseudo--$\fg_0$--structure of type $\fg_-$ on the principal $G_{0,I}$--bundle $\ba_0$ of complex frames of $\mathcal{D}/\mathcal{K}$ is obtained by normalizing the components of homogeneity $0,1$ w.r.t. to the gradings given by $\Sigma_1,\Sigma_2$ of complexification of $d(\si^*\om)+[\si^*\om,\si^*\om]$ evaluated on elements of $\si^*\om^{-1}(\fg)$.

We need to emphasize that the infinitesimal pseudo--$\fg_0$--structure of type $\fg_-$ does not allow, in general, to reconstruct the whole complex structure $\mathcal{I}$ on $\mathcal{D}$, it only provides the complex structures on $\mathcal{K}$ and $\mathcal{D}/\mathcal{K}$. This is reflected in the normalization conditions in the construction of $\om$, where we need to assume that  the restriction of $\si^*\om$ to ${T^{-1}\ba_0}\to \fg^{-1}/(\fg_{0,I}\oplus \fp_+)$ is complex linear and then proceed with further normalization conditions on components of (complexification of) $d(\si^*\om)+[\si^*\om,\si^*\om]$. Let us remark that in the case $\frak{sp}(2N,\mathbb{R}),\Sigma_1=\{1\}$, this provides an additional reduction of the slight generalization of the first Tanaka prolongation to the infinitesimal pseudo--$\fg_0$--structure of type $\fg_-$ described in \cite[Section 4.1]{Gr19}.

At first glance, the properties (1) are the only difference between the pseudo--$G_0$--structure and the infinitesimal pseudo--$\fg_0$--structure, i.e., the $\fg_0$--action does not have to be integrable to $G_0$--action. However, we need to emphasize that the normalization conditions mentioned in the above constructions of   $\om$ and $\tilde \om$ are not compatible, in general.

\section{Underlying parabolic geometries and the CR geometry on the correspondence space}\label{kap2under}

There is a known construction of infinitesimal pseudo--$\fg_0$--structure of type $\fg_-$ from a parabolic geometry of type $(G,P)$. If we start with the normal, regular parabolic geometry $(\ba\to N,\tilde \om)$ of type $(G,P)$ corresponding to one of the cases from Tables  \ref{realclasA} and \ref{realclasB}, where $P=G_0\exp(\fp_+)$, then $$\ba_0:=\ba/\exp(\fp_+)$$ is principal $G_{0,I}$--bundle over $$M:=\ba/G_{0,I}\exp(\fp_+)$$ and $\si^*\tilde \om$ induces a  pseudo--$\fg_0$--structure of type $\fg_-$ on $\ba_0\to M$ that is independent of the choice of Weyl structure $\si$. However, in general, this determines only an almost CR geometry with second--order Levi--Tanaka algebra $(\fg_-,I,\fk)$ on $M$, which we call the \emph{almost CR structure on the correspondence space} $M$.

\begin{thm*}\label{harm}
The 2--nondegenerate almost CR structure $(M,\mathcal{D},\mathcal{\tilde I})$ on the correspondence space $\ba/G_{0,I}\exp(\fp_+)$ of a regular, normal parabolic geometry $(\ba\to N,\tilde\om)$  of type $(G,P)$ corresponding to one of the cases from the Tables  \ref{realclasA} and \ref{realclasB} is (formally) integrable if and only if one of the following claims holds:

\begin{enumerate}
\item the parabolic geometry is flat and there,
\item the type of parabolic geometry is listed in the Table \ref{realclasC} and the harmonic curvature is nontrivial only in the specified components.
\end{enumerate}

If (1) holds, then there are no fundamental invariants and if (2) holds, then the fundamental invariants of the CR geometry are provided by the component of the complexification of the admissible harmonic curvature in the $\fg_{0,0}$--submodule containing the of lowest weight vector.
\end{thm*}
\noindent
\begin{table}[H]\caption{Classification of fundamental invariants corresponding to harmonic curvatures}\label{realclasC}
\begin{tabular}{|c|c|c|}
\hline
$\mathfrak{g}$ & $\Sigma_1$ & harm. cur.\\
$\fg_{0,I}'$& $\Sigma_2$ &  restrictions \\
\hline
$\mathfrak{sl}(4,\mathbb{R})$&$\{\alpha_1,\alpha_{3}\}$ &$(1,2),(3,2)$ \\
$\frak{so}(2)$ &$\{\alpha_{2}\}$& \\
\hline
$\mathfrak{sl}(n+1,\mathbb{R})$&$\{\alpha_1,\alpha_{3}\}$ &$(1,2)$ \\
$\frak{so}(2)$ &$\{\alpha_{2}\}$&$n>3$ \\
\hline
$\mathfrak{su}(p,4-p)$&$\{\alpha_1,\alpha_{3}\}$ & $(1,2),(3,2)$ \\
$\frak{u}(p,2-p)$&$\{\alpha_2\}$ & \\
\hline
$\mathfrak{so}(3,4)$& $\{\alpha_1,\alpha_{3}\}$ &$(3,2)$ \\
 $\frak{so}(2)$&$\{\alpha_{2}\}$ &\\
\hline
$\mathfrak{so}(p,q)$ or $\mathfrak{so}^*(2n)$& $\{\alpha_2\}$ &$(2,1)$ \\
$\frak{so}(2)$& $\{\alpha_1\}$&$p+q>5$\\
\hline
$\mathfrak{sp}(2n,\mathbb{R})$ or $\mathfrak{sp}(p,p+1)$ & $\{\alpha_{n-1}\}$&$(n-1,n)$  \\
$\frak{so}(2)$ &$\{\alpha_n\}$&$n=2p+1$\\
\hline
$\mathfrak{g}_2(2)$ & $\{\alpha_1\}$&$(2,1)$\\
$\frak{so}(2)$& $\{\alpha_2\}$& \\
\hline
\end{tabular}
\end{table}
\begin{proof}
Let us start with regular, normal parabolic geometry $(\ba\to N,\tilde\om)$  of type $(G,P)$ corresponding to one of the cases from the Tables  \ref{realclasA} and \ref{realclasB}. Then on the correspondence space $M=\ba/G_{0,I}\exp(\fp_+)$ with second--order Levi--Tanaka algebra $(\fg_-,I,\fk)$, we have the distribution $$\mathcal{D}\cong \ba\times_{G_{0,I}\exp(\fp_+)}(\fg^{-1}/(\fg_{0,I}\oplus \fp_+))$$ and complex structure $\tilde{\mathcal{I}}$ induced via $\tilde\om$ by the complex structures on $\fg_{-1}\oplus \fk$. This is only an almost CR structure, because the bundles $$\mathcal{D}^{10}\cong \ba\times_{G_{0,I}\exp(\fp_+)}((\fg_{-1,-1}\oplus \fg_{0,-1}\oplus \fg_{0,0}\oplus \fp_+\otimes \C)/(\fg_{0,0}\oplus \fp_+\otimes \C)$$ and $$\mathcal{D}^{01}\cong \ba\times_{G_{0,I}\exp(\fp_+)}((\fg_{-1,0}\oplus \fg_{0,1}\oplus \fg_{0,0}\oplus \fp_+\otimes \C)/(\fg_{0,0}\oplus \fp_+\otimes \C)$$ do not have to be integrable.

Let us recall that $Ker(\partial^*)/Im(\partial^*)$ decomposes into irreducible $\fg_0$--modules and the lowest weight space $$\fg_{-\alpha_i}^*\otimes \fg_{-s_{\alpha_j}\alpha_i}^*\otimes\fg_{-s_{\alpha_j}s_{\alpha_i}\lambda}$$ of each  $\fg_0$--module is determined by an ordered pair of numbers $(i,j)$, where $s$ is the reflexion in the root lattice and $\lambda$ is the highest root. Since we view the complexification harmonic curvature on the fibres isomorphic to $G_0/G_{0,I}$, the component of harmonic curvature in $\fg_{0,0}$--submodule containing the lowest weight vector is enough to determine the full harmonic curvature.

The regularity restricts the possible pairs $(i,j)$ and the full classification can be found for example in \cite[Appendix C]{KT17}. Thus we can compute the bigrading components of the lowest weight space $\fg_{-\alpha_i}^*\otimes \fg_{-s_{\alpha_j}\alpha_i}^*\otimes\fg_{-s_{\alpha_j}s_{\alpha_i}\lambda}$ for the cases in our tables. There are only several cases that can appear:

If there is nontrivial haramonic curvature in components $\fg_{-1,0}^*\otimes \fg_{-1,0}^*\otimes \fg_{-1,-1}$ or $\fg_{-1,0}^*\otimes \fg_{-1,0}^*\otimes \fg_{0,-1}$, then there are points in the fiber that have curvature in these $\fg_{0,0}$--submodules, which are clearly obstructions to integrability.

Otherwise, if there is nontrivial haramonic curvature in components $\fg_{-1,0}^*\otimes \fg_{a,-1}^*\otimes \fg_{b,c}$ for $a<0$, then it follows by the general results on the reconstruction of the curvature $\kappa$ from the harmonic curvature via splitting operators, c.f. \cite[Section 3.3]{CS17}, that the components of curvature $\fg_{-1,0}^*\otimes \fg_{-1,0}^*\otimes \fg_{-1,-1}$ or $\fg_{-1,0}^*\otimes \fg_{-1,0}^*\otimes \fg_{0,-1}$ vanish and thus the CR geometry is (formally) integrable.
\end{proof}

Therefore if we prove an existence and uniqueness of the underlying parabolic geometry, then we can compare the CR geometry $(M,\mathcal{D},\mathcal{I})$ with the CR geometry on the correspondence space. For this we need to distinguish the $\frak{sp}(2N,\mathbb{R}),\Sigma_1=\{1\}$ case from the other cases.

\begin{prop*}\label{under}
Let $(M,\mathcal{D},\mathcal{I})$ be a 2--nondegenerate CR geometry with second--order Levi--Tanaka algebra $(\fg_-,I,\fk)$ corresponding to one of the cases from Tables  \ref{realclasA} and \ref{realclasB} different from $\frak{sp}(2N,\mathbb{R}),\Sigma_1=\{1\}$. Then for each $x\in M$, there is a neighborhood $U$ of $x$ and  Cartan geometry $(\ba\to N,\tilde\om)$ of type $(G,P)$ on the leaf space $U\to N$ with leafs corresponding to integral manifolds of the Levi kernel $\mathcal{K}$ such that the infinitesimal pseudo--$\fg_0$--structure of type $\fg_-$ on $U$ is isomorphic to an open subset of the infinitesimal pseudo--$\fg_0$--structure of type $\fg_-$ on the correspondence space $\ba/G_{0,I}\exp(\fp_+)$ induced by $\tilde \om$.
\end{prop*}
\begin{proof}
Firstly, we can choose a neighborhood $U$ of $x$ in such a way that the leaf space $U\to N$ exists and the bundle $\ba_0$ of complex frames of $\mathcal{D}/\mathcal{K}$ is trivial over $U$. Further, we shrink $U$ such that fundamental fields of the infinitesimal $\fg_0$--action on $\ba_0$ identify $\ba_0$ over $U$ with an open subset of $N\times G_0$. Using the $G_0$--action we obtain a pseudo--$G_0$--structure $N\times G_0$ of type $\fg_-$ on $N$ that coincides with the infinitesimal pseudo--$\fg_0$--structure of type $\fg_-$ on $\ba_0$. So as we reviewed in the previous section, we obtain a unique underlying parabolic geometry of type $(G,P)$ and  the description of correspondence spaces from \cite[Sections 1.5.13 and 1.5.14]{parabook} translate into the statement of the proposition.
\end{proof}

Let us emphasize that the Cartan geometry $(\ba\to N,\tilde\om)$ of type $(G,P)$ from Proposition \ref{under} provides the same complex structure $\mathcal{\tilde I}$ on every $T^{-1}U$ and thus a complex structure $\mathcal{\tilde I}$ on $\mathcal{D}$.

\section{Fundamental invariants in the general case}\label{genc}

In this section, we compare the complex structure $\mathcal{I}$ on $\mathcal{D}$ with the complex structure $\mathcal{\tilde I}$ on $\mathcal{D}$ provided by the underlying parabolic geometry from the previous section. This allows us to obtain the fundamental invariants for all cases except $\frak{sp}(2N,\mathbb{R}),\Sigma_1=\{1\}$.

\begin{thm*}\label{fin1}
Let $(M,\mathcal{D},\mathcal{I})$ be a 2--nondegenerate CR geometry with second--order Levi--Tanaka algebra $(\fg_-,I,\fk)$ corresponding to one of the cases from the Tables  \ref{realclasA} and \ref{realclasB} different from $\frak{sp}(2N,\mathbb{R}),\Sigma_1=\{1\}$. If $\fg\neq \mathfrak{g}_2(2)$, then

\begin{enumerate}
\item $\mathcal{I}$ and $\mathcal{\tilde I}$ coincide.
\item There is a Cartan connection $\om$ of type $(G,G_{0,I}\exp(\fp_+))$ solving the equivalence problem of CR geometries with second--order Levi--Tanaka algebra $(\fg_-,I,\fk)$.
\item The Cartan connection can be non flat in cases from Table \ref{realclasC} that characterizes all fundamental invariants.
\item The second--order Levi--Tanaka algebra $(\fg_-,I,\fk)$ is not generic.
\end{enumerate}

If $\fg= \mathfrak{g}_2(2)$, then
\begin{enumerate}
\item the difference $\mathcal{I}-\mathcal{\tilde I}$ is uniquely determined by a section $G$ of $$\ba_0\times_{G_{0,I}} \fg_{-1,-1}^*\otimes \fg_{0,1}$$ that depends on one complex function $G^{\bar 1}_{11}$ that is antiholomorphic along the leaves of the Levi kernel.
\item The function $G^{\bar 1}_{11}$ and the harmonic curvature of the underlying $(2,3,5)$ geometry are the only fundamental invariants.
\item This is a generic case of seven--dimensional submanifolds in $\C^{5}$ with two--dimensional Levi kernel.
\end{enumerate}
\end{thm*}
\begin{proof}
Suppose that $\mathcal{\tilde I}$ is the complex structure on $\mathcal{D}$ corresponding to the underlying regular parabolic geometry $(\ba\to N,\tilde\om)$ of type $(G,P)$. Without loss of generality, we assume that $\om$ is an absolute parallelism on $\ba$ solving the equivalence problems for the 2--nondegenerate CR geometry with second--order Levi--Tanaka algebra $(\fg_-,I,\fk)$ constructed in \cite[Section 4.2]{Gr19} for a  pullback by a Weyl structure $\si$ of the underlying parabolic geometry. We compare the complexifications $\si^*\om$ and $\si^*\tilde \om$ with respect to the bigrading $\fg_{a,b}$ given by the sets $\Sigma_1,\Sigma_2$ from the Tables  \ref{realclasA} and \ref{realclasB}.

We have shown in \cite[Section 4.2]{Gr19} that $\si^*(\tilde \om-\om)$ has positive homogeneity w.r.t. the first of the gradings and thus $\mathcal{I}$ and $\mathcal{\tilde I}$ coincide on $\mathcal{K}$ and $\mathcal{D}/\mathcal{K}$. If we consider a complex frame $L_i$ of $$\si^*\tilde\om^{-1}(\fg_{-1,-1}\oplus (\fp_+\otimes \C))$$ on complefication of $T\ba_0$ given by a choice of basis of $\fg_{-1,-1}$, then we can write each element of $$\si^*\tilde\om^{-1}(\fg_{0,1}\oplus (\fp_+\otimes \C))$$ on complefication of $T\ba_0$ as a linear combination of nilpotent elements $\bar K^i_{\bar j}$ corresponding to elements of $\frak{gl}(\fg_{-1,0}\oplus \fg_{-1,-1})$ mapping $L_i$ onto $\bar L_{\bar j}$, where $\bar L_{\bar j}$ is the conjugated basis of  $\si^*\tilde\om^{-1}(\fg_{-1,0}\oplus (\fp_+\otimes \C))$. Then we can write $$L_i+G(L_i)=L_i+G^{\bar j}_{ik}\bar K^k_{\bar j}$$ for the corresponding complex frame of  $\si^*\om^{-1}(\fg_{-1,-1}\oplus (\fp_+\otimes \C))$, where $G^{\bar j}_{ik}$ defines a section $G$ of $$\ba_0\times_{G_{0,I}} \fg_{-1,-1}^*\otimes \fg_{0,1}\subset \fg_{-1,-1}^*\otimes\fg_{-1,-1}^*\otimes \fg_{-1,0},$$ because $[L_a,\bar K^a_{\bar j}]=\bar L_{\bar j}$. In order to have the integrability of the CR geometry, we need to check that $[L_a+G(L_a),L_b+G(L_b)]$ and $[L_a+G(L_a),K^{\bar b}_c]$ are contained in $$\si^*\om^{-1}(\fg_{-1,-1}\oplus \fg_{0,-1}\oplus \fg_{0,0}\oplus (\fp_+\otimes \C)).$$ If we decompose $[L_a+G(L_a),L_b+G(L_b)]$ as
\begin{align*}
&([L_a,L_b]+G([L_a,L_b])-(G^{\bar j}_{ba}-G^{\bar j}_{ab})\bar G(\bar L_{\bar j}))\\
&+(G^{\bar j}_{ba}-G^{\bar j}_{ab})(\bar L_{\bar j}+\bar G(\bar L_{\bar j}))+((L_a(G^{\bar j}_{bk})-L_b(G^{\bar j}_{ak}))\bar K^k_{\bar j}-G([L_a,L_b])),
\end{align*}
then only the first line is not an obstruction to integrability. The first term on the second line is in $\si^*\om^{-1}(\fg_{-1,0}\oplus(\fp_+\otimes \C))$ provides an algebraic  obstruction for integrability and the second term on the second line is in  $\si^*\om^{-1}( \fg_{0,1}\oplus (\fp_+\otimes \C))$ provides a differential obstruction for integrability.  Similarly, $$[L_a+G(L_a),K^{\bar b}_c]=K^{\bar b}_c(G^{\bar j}_{ak})\bar K^k_{\bar j}$$ is the final obstruction for integrability.

Since the pairing $\fg_{0,1}\otimes \fg_{-1,-1}\to \fg_{-1,0}$ given by Lie bracket is non--degenerate and $\fg_{0,1}$ is irreducible $\fg_{0,0}$--module, the conditions $$G^{\bar j}_{ba}-G^{\bar j}_{ab}=0$$ provide strong restrictions on the CR geometry. Firstly, there should be no semisimple component of $\fg_{0,0}$ that acts only on $\fg_{-1,-1}$, and $\fg_{-1,-1}$ should also be irreducible $\fg_{0,0}$--module. Moreover, $dim(\fg_{0,1})$ should be larger or equal to $dim(\fg_{-1,-1})$. 

The cases in Table \ref{realclasA} that satisfy these conditions are only the cases with $\fg=\frak{sp}(2N,\mathbb{R}),\Sigma_1=\{1\}$.

The cases in Table \ref{realclasB} that satisfy these conditions are only the cases with $\fg=\mathfrak{g}_2(2),\Sigma_1=\{1\}$. We can observe that $G^{\bar j}_{ba}-G^{\bar j}_{ab}=0$ is satisfied in this case due to dimensional reasons.

Now, we discuss the genericity of the second--order Levi--Tanaka algebra $(\fg_-,I,\fk)$ corresponding to the cases from the Tables  \ref{realclasA} and \ref{realclasB}. Firstly, we recall that $\fg_-$ is generic in our situation only if:
\begin{enumerate}
\item $k=2$ and $dim(\fg_{-2})=1$,
\item $k=3$ and $dim(\fg_{-1})=2$, $dim(\fg_{-2})=1$, $dim(\fg_{-3})=2$.
\end{enumerate}
Further, let us fix $\fg_-$ and consider $\tilde \fg_0$ consisting of grading preserving derivations of $\fg_-$. We need to look on $\tilde G_{0,I}$--orbits of $\fk$ inside of the space of complex antilinear derivation in $\tilde \fg_0$. In all cases except
\begin{enumerate}
\item $k=2$ and $dim(\fg_{-2})=1$ other than $(Sp(2N,\mathbb{R}),P_{1})$, 
\item $(Sl(N+1,\mathbb{R}),P_{1,1+2s})$,
\end{enumerate}
$\tilde \fg_0=\fg_0$ and the genericity of $\fk$ is trivially satisfied.

In the case $(Sl(N+1,\mathbb{R}),P_{1,2s})$, $\tilde \fg_0\neq \fg_0$, because there are additional complex antilinear derivations in $\tilde \fg_0$ in a different $\fg_0$--module than $\fg_0$ (which contains all complex linear derivations) and thus $\fk$ is not generic in these cases.

In the case $k=2$ and $dim(\fg_{-2})=1$ is $\tilde \fg_0= \frak{csp}(2N-2,\mathbb{R})$, and $\fg_{0,I}'=\frak{u}(p,N-1-p)$ does not act transitively on the space of complex antilinear derivations in $\tilde \fg_0$. Therefore, with exception of the case $(Sp(2N,\mathbb{R}),P_{1})$, these cases are not generic.

This proves the first claim.

For the second claim, we observe that the only (due to dimensional reasons) integrability condition is $$K^{\bar 1}_1(G^{\bar 1}_{11})=0,$$ i.e., the complex function $G^{\bar 1}_{11}$ is antiholomorphic along the leaves of $\mathcal{K}$. If we write $$\si^*\om=(\id+G+\bar G+F)\circ \si^*\tilde\om,$$ where $F:\fg_-\oplus \fk\to \fg$ corresponds to change of normalization, then $$d(\si^*\om)+[\si^*\om,\si^*\om]=d((\id+G+\bar G+F)\circ \si^*\tilde\om)+[(\id+G+\bar G+F)\circ \si^*\tilde\om,(\id+G+\bar G+F)\circ \si^*\tilde\om]$$ provides the change of the structure functions. If we insert $L_i$ and $\si^*\tilde\om^{-1}(Y)$ for $Y\in \fg_{-3,-2}$ in the the complexification of the structure function, then we get $G^{\bar 1}_{11}[\si^*\om(\bar K_{\bar 1}^1),Y]$ in $\fg_{-3,-1}$, because 
\begin{gather*}
d\si^*\tilde\om(L_i,Y)+[ \si^*\tilde\om(L_i), \si^*\tilde\om(Y)]=0,\\
d((G+\bar G+F)\circ \si^*\tilde\om)\in \fg^{-2}\otimes \C,\\
[\si^*\tilde\om(L_i),(G+\bar G+F)\circ \si^*\tilde\om(Y)]\subset \fg_{-3,-2}\oplus \fg^{-2}\otimes \C\\
\end{gather*}
 and $[(G+\bar G+F)\circ \si^*\tilde\om(L_i),(\id+G+\bar G+F)\circ \si^*\tilde\om(Y)]=[\si^*\tilde\om(G(L_i)),Y]=G^{\bar 1}_{11}[\si^*\om(\bar K_{\bar 1}^1),Y]$ modulo  $\fg_{-3,-2}\oplus \fg^{-2}\otimes \C$. Therefore $G^{\bar 1}_{11}$ is the additional fundamental invariant of the CR geometry.
\end{proof}

Let us provide an example of the CR geometry with $\fg= \mathfrak{g}_2(2)$ and nontrivial fundamental invariant different from the harmonic curvature. We start with the flat model for which we can obtain the following defining equations using \cite[Proposition 3.7]{Gr19}:

\begin{align*}
\Re(w_1)&={\frac {- \left( { \bar z_1} {{ z_2}}^{2}+{{ \bar z_2}}^{2}{z_1
}+2 { \bar z_2} { z_2} \right)  \sqrt{2}}{2({ \bar z_1} { z_1}-1)}}\\
\Re(w_2)&={\frac {3\sqrt{2} \left( -{ \bar z_1} { z_2}+{ \bar z_2} { z_1
} \right) { \Im(w_1)}}{{4 \bar z_1} { z_1}}}i-{\frac {2 \left( { z_1}
-{ \bar z_1} \right) { \Im(w_2)}}{{ \bar z_1} { z_1}}}i+{\frac { \left( 
 \left( { z_1}+2 \right) { \bar z_1}+2 { z_1} \right) { \Im(w_3)}}{{
 \bar z_1} { z_1}}}\\
 &+\frac {-{{ \bar z_1}}^{3}{ z_1} {{ z_2}}
^{3}+ \left( 3 { z_2} { \bar z_2}  \left( { \bar z_2}+{ z_2}
 \right) {{ z_1}}^{2}+ \left( 6 { \bar z_2} {{ z_2}}^{2}+4 {{ 
z_2}}^{3} \right) { z_1}+5 {{ z_2}}^{3} \right) {{ \bar z_1}}^{2}}{{4 \bar z_1} { z_1}  \left( { \bar z_1} { z_1}-1 \right) ^
{2}}\\
&+\frac{-
 \left( {{ \bar z_2}}^{2}{{ z_1}}^{3}+ \left( -4 {{ \bar z_2}}^{2}-6 {
 \bar z_2} { z_2} \right) {{ z_1}}^{2}-9 { z_2}  \left( { 
\bar z_2}+{ z_2} \right) { z_1}-6 {{ z_2}}^{2} \right) { \bar z_2} {
 \bar z_1}+5 {{ \bar z_2}}^{3}{{ z_1}}^{2}+6 {{ \bar z_2}}^{2}{ z_1} 
{ z_2}}{{4 \bar z_1} { z_1}  \left( { \bar z_1} { z_1}-1 \right) ^
{2}}\\
\Re(w_3)&= {\frac {-3\sqrt {2} \left( { \bar z_1} { z_2}+{ \bar z_2} { z_1}
 \right) { \Im(w_1)}}{{4 \bar z_1} { z_1}}}-{\frac { \left(  \left( { 
z_1}-2 \right) { \bar z_1}-2 { z_1} \right) { \Im(w_2)}}{{ \bar z_1} { 
z_1}}}+{\frac {2\left( { z_1}-{ \bar z_1} \right) { \Im(w_3)}}{{ 
\bar z_1} { z_1}}}i\\
&+\frac {i{ z_1} {{ z_2}}^{3}{{ \bar z_1}}^{3}
+i{ z_2}  \left(3 { \bar z_2}  \left( -{ \bar z_2}+{ z_2} \right) 
{{ z_1}}^{2}+ \left( 4 {{ z_2}}^{2}-2 { z_2} { \bar z_2}
 \right) { z_1}-5 {{ z_2}}^{2} \right) {{ \bar z_1}}^{2}}{{4 \bar z_1} { z_1}  \left( { \bar z_1} { z_1}-1
 \right) ^{2}}\\
 &+\frac{ -i
 \left( {{ \bar z_2}}^{2}{{ z_1}}^{3}+ \left( 4 {{ \bar z_2}}^{2}-6 {
 z_2} { \bar z_2} \right) {{ z_1}}^{2}-9 { z_2}  \left({ z_2} -{ 
\bar z_2} \right) { z_1}+6 {{ z_2}}^{2} \right) { \bar z_2} {
 \bar z_1}+i \left(5 { \bar z_2} { z_1}+6 { z_2} \right) {{ 
\bar z_2}}^{2}{ z_1}}{{4 \bar z_1} { z_1}  \left( { \bar z_1} { z_1}-1
 \right) ^{2}}
\end{align*}

We pick the following complex frame of $\mathcal{D}/\mathcal{K}$:

\begin{align*}
\mathcal{L}_1&:=-\frac{\bar z_1 z_2+\bar z_2}{\bar z_1 z_1-1}\frac{\partial}{\partial z_2}
+\frac{-i \sqrt{2} (\bar z_1^2 z_2^2+2 \bar z_1 \bar z_2 z_2+\bar z_2^2) }{(\bar z_1 z_1-1)^2}\frac{\partial}{\partial \Im(w_1)}\\
&+[-3\sqrt{2}\Im(w_1)\frac{
 \bar z_2 + \bar z_1^3  z_1^2 z_2-2 \bar z_1^2 z_1 z_2
-2 \bar z_1 \bar z_2 z_1+\bar z_1^2 \bar z_2 z_1^2+ \bar z_1 z_2}{16(\bar z_1 z_1-1) (\bar z_1^2 z_1^2-2 \bar z_1 z_1+1)}\\
&+3i\frac{\bar z_1^3 z_1 z_2^3-\bar z_1^2 \bar z_2^2 z_1^2 z_2+4 \bar z_1^3 z_2^3-\bar z_1^2 \bar z_2 z_1 z_2^2-\bar z_1 \bar z_2^3 z_1^2+12 z_2^2 \bar z_2 \bar z_1^2-3 \bar z_1^2 z_2^3-3 \bar z_1 \bar z_2^2 z_1 z_2}{16(\bar z_1 z_1-1) (\bar z_1^2 z_1^2-2 \bar z_1 z_1+1)}\\
&+3i\frac{12 z_2 \bar z_2^2 \bar z_1-5 \bar z_1 \bar z_2 z_2^2-\bar z_2^3 z_1+4 \bar z_2^3-2 \bar z_2^2 z_2}{16(\bar z_1 z_1-1) (\bar z_1^2 z_1^2-2 \bar z_1 z_1+1)}]\frac{\partial}{\partial \Im(w_2)}\\
&+[3i\sqrt{2}\Im(w_1)\frac{
- \bar z_1^3  z_1^2 z_2- \bar z_1^2 \bar z_2 z_1^2+2 \bar z_1^2  z_1 z_2- \bar z_1  z_2- \bar z_2 +2  \bar z_1 \bar z_2 z_1}{16(\bar z_1 z_1-1) (\bar z_1^2 z_1^2-2 \bar z_1 z_1+1)}\\
&-3\frac{-2 \bar z_2^2 z_2-4 \bar z_1^3 z_2^3-5 \bar z_1 \bar z_2 z_2^2-3 \bar z_1^2 z_2^3-\bar z_1^2 \bar z_2^2 z_1^2 z_2+\bar z_1^3 z_1 z_2^3-12 \bar z_1 \bar z_2^2 z_2-\bar z_1 \bar z_2^3 z_1^2-4 \bar z_2^3}{16(\bar z_1 z_1-1) (\bar z_1^2 z_1^2-2 \bar z_1 z_1+1)}\\
&-3\frac{-12 \bar z_1^2 \bar z_2 z_2^2-3 \bar z_1 \bar z_2^2 z_1 z_2-\bar z_1^2 \bar z_2 z_1 z_2^2-\bar z_2^3 z_1}{16(\bar z_1 z_1-1) (\bar z_1^2 z_1^2-2 \bar z_1 z_1+1)}]\frac{\partial}{\partial \Im(w_3)}
\end{align*}

and the following complex frame of $\mathcal{K}$:

\begin{align*}
\mathcal{K}^{\bar 1}_1&:=-\frac{(\bar z_1 z_2+\bar z_2) (\bar z_1 z_1-1)}{\bar z_2 z_1+z_2}\frac{\partial}{\partial z_1}
-\frac{(\bar z_1 z_2+\bar z_2)^2}{\bar z_2 z_1+z_2}\frac{\partial}{\partial z_2}\\
&-i\sqrt{2}\frac{(\bar z_1 z_2+\bar z_2)^3}{2(\bar z_2 z_1+z_2) (\bar z_1 z_1-1)}\frac{\partial}{\partial \Im(w_1)}\\
&+i(\bar z_1 z_2+\bar z_2)[ \Im(w_3)\frac{\bar z_1  z_1-1}{16(\bar z_2 z_1+z_2) z_1}+i\Im(w_2)\frac{\bar z_1 z_1-1}{2(\bar z_2 z_1+z_2) z_1}+\frac{3i\sqrt{2}\Im(w_1)}{16 z_1}\\
&+\frac{12 \bar z_1^2 \bar z_2 z_1 z_2^2+3 \bar z_1 \bar z_2^2 z_1^2 z_2-3 \bar z_1 \bar z_2 z_1 z_2^2+12 \bar z_1 \bar z_2^2 z_1 z_2+4 \bar z_2^3 z_1+4 \bar z_1^3 z_1 z_2^3-9 \bar z_2^2 z_1 z_2}{16(\bar z_2 z_1+z_2) z_1 (\bar z_1^2 z_1^2-2 \bar z_1 z_1+1)}\\
&+\frac{3 \bar z_1^2 z_1 z_2^3+\bar z_1 \bar z_2^3 z_1^3-5 \bar z_1 z_2^3-6 \bar z_2 z_2^2-3 \bar z_2^3 z_1^2+3 \bar z_1^2 \bar z_2 z_1^2 z_2^2}{16(\bar z_2 z_1+z_2) z_1 (\bar z_1^2 z_1^2-2 \bar z_1 z_1+1)}]\frac{\partial}{\partial \Im(w_2)}\\
&+i(\bar z_1 z_2+\bar z_2)[- \Im(w_2)\frac{\bar z_1  z_1-1}{2(\bar z_2 z_1+z_2) z_1}+i\Im(w_3)\frac{\bar z_1  z_1-1}{2(\bar z_2 z_1+z_2) z_1}-\frac{3 \sqrt{2} \Im(w_1)}{16 z_1}\\
&+i\frac{-4\bar z_1^3 z_1 z_2^3-9 \bar z_2^2 z_1 z_2+3 \bar z_1^2 z_1 z_2^3-4 \bar z_2^3 z_1-6z_2^2 \bar z_2-5 \bar z_1 z_2^3-3 \bar z_2^3 z_1^2}{16(\bar z_2 z_1+z_2) z_1 (\bar z_1^2 z_1^2-2 \bar z_1 z_1+1)}\\
&+i\frac{-12 \bar z_1 \bar z_2^2 z_1 z_2+3 \bar z_1^2 \bar z_2 z_1^2 z_2^2+3 \bar z_1 \bar z_2^2 z_1^2 z_2-3 \bar z_1 \bar z_2 z_1 z_2^2+ \bar z_1 \bar z_2^3 z_1^3-12\bar z_1^2 \bar z_2 z_1 z_2^2}{16(\bar z_2 z_1+z_2) z_1 (\bar z_1^2 z_1^2-2 \bar z_1 z_1+1)}]\frac{\partial}{\partial \Im(w_3)}.
\end{align*}

It is not hard to verify that $\mathcal{L}_1$ and $ \mathcal{\bar L}_{\bar 1}$ generate graded Lie algebra isomorphic to $\fg_-$ of $(2,3,5)$ distribution and that $\mathcal{K}^{\bar 1}_1$ maps $\mathcal{\bar L}_{\bar 1}$ on $\mathcal{L}_1$ (modulo Levi kernel).

Now, we change the complex structure by taking $$\mathcal{L}_1+G^{\bar 1}_{11}\mathcal{\bar K}^1_{\bar 1}$$ as a complex frame of $\mathcal{D}/\mathcal{K}$. In this chosen frame, the integrability condition of the CR geometry takes form $$K^{\bar 1}_1(G^{\bar 1}_{11})=\mathcal{K}^{\bar 1}_{1}(G^{\bar 1}_{11})+\frac{2\bar z_1(\bar z_1z_2+\bar z_2)}{\bar z_2z_1+\bar z_1}G^{\bar 1}_{11}=0,$$ where we determined $K^{\bar 1}_1$ using \cite[Proposition 4.2]{Gr19}. The solver of differential equations in Maple software did not provide general solution of this PDE. Assuming that $G^{\bar 1}_{11}$ does not depend on the w's we get a class of solutions
$$G^{\bar 1}_{11}=\frac{F(\bar z_1,\bar z_2,\frac{\bar z_1 z_2+\bar z_2}{\bar z_1(\bar z_1 z_1-1)})}{(\bar z_1 z_1-1)^2}$$
depending on one function $F$ of three variables. By construction, this is the only fundamental invariant of this CR geometry.

\section{Fundamental invariants in $\frak{sp}(2N,\mathbb{R}),\Sigma_1=\{1\}$ case}\label{KapSp}

In the $\frak{sp}(2N,\mathbb{R}),\Sigma_1=\{1\}$ case (independently of the choice of $\fg_{0,I}'=\frak{u}(p,N-1-p)$), we can (locally) construct a pseudo--$G_0$--structure of type $\fg_-$ on the leaf space with leafs corresponding to integral manifolds of the Levi kernel $\mathcal{K}$. However, this is not enough to determine a contact projective geometry, i.e., there are many possible underlying parabolic geometries of type $(Sp(2N,\mathbb{R}),P_{1})$. Since the fundamental invariants for the case $N=2$ are well--known, see \cite{IZ13,Poc13,KK19}, we consider only the case $N>2$. 

The bigrading of $\frak{sp}(2N,\C)$ takes the following $(1,N-1,N-1,1)$--block matrix form:

\begin{gather*}
\left[ \begin{array}{cccc}
W&\bar L^{*\bar i}&L^{*i}&J^*_{i\bar i}\\
\bar L_{\bar i}&V^{\bar i}_{\bar j}&\bar K_{\bar i}^j& L^*_i\\
L_i&K_i^{\bar j}&-V^j_i& -\bar L^*_{\bar i}\\
J^{i\bar i}&L^i& -\bar L^{\bar i}& -W
\end{array}\right]\subset \left[ \begin{array}{cccc}
\fg_{0,0}&\fg_{1,0}&\fg_{1,1}&\fg_{2,1}\\
\fg_{-1,0}&\fg_{0,0}&\fg_{0,1}&*\\
\fg_{-1,-1}&\fg_{0,-1}&*& *\\
\fg_{-2,-1}&*&*& *
\end{array}\right],
\end{gather*}
where $W,J,J^*\in \C$, $L,\bar L,L^*,\bar L^*\in \C^{N-1}$, $K\in S^2\C^{N-1}$, $V\in  \C^{N-1}\otimes  \C^{N-1}$ and the indices indicate, how these vectors are represented as rows, columns, and matrices (i.e., the $*$ entries in the second matrix are dependent on the other entries). In what follows, we do not need know, how $\frak{sp}(2N,\mathbb{R})$ sits as a real form in $\frak{sp}(2N,\C)$. However, let us emphasize that this real form depends on the signature of the Levi form, i.e., on $\fg_{0,I}'=\frak{u}(p,N-1-p)$.

Let emphasize that for a chosen Weyl structure, the entries $J^{i\bar i},J^*_{i\bar i}$ represent the Levi--form that can be used to lower and raise indices, i.e., $K_i^{\bar j}J^*_{j\bar j}=K_{ij}=K_{ji}=K_j^{\bar i}J^*_{i\bar i}$ is a symmetric matrix. Therefore the algebraic integrability condition from proof of Theorem \ref{fin1} can be interpreted such that $G_{ik}^{\bar j}J^*_{j\bar j}=G_{ijk}$ is totally symmetric.

\begin{thm*}\label{fin2}
Let $(M,\mathcal{D},\mathcal{I})$ be a 2--nondegenerate CR geometry with second--order Levi--Tanaka algebra $(\fg_-,I,\fk)$ corresponding to $\frak{sp}(2N,\mathbb{R}),\Sigma_1=\{1\},N>2$ case. Then $\bar K_{\bar l\bar m}(G_{ijk})$ decomposes into irreducible submodules of $S^3\fg_{-1,-1}^*\otimes S^2 \fg_{-1,0}^*$:
\begin{enumerate}
\item to (a complexification of) a fundamental invariant in the highest weight component,
\item to a section $\om^{-1,0}_{0,1}$ of $\ba_0\times_{G_{0,I}}\fg_{-1,0}^*\otimes \fg_{0,1}$ in the other nonzero components,
\end{enumerate}
where $\bar K^l_{\bar m}$ are the constant vector fields for the Maurer--Cartan form on (an open subset of) $\ba_0$ provided by subgroup $\exp(\fg_-)\rtimes G_{0,I}$ of the automorphisms of a flat contact projective geometry acting (locally) transitively on $\ba_0$.

The remaining (complexification of)  fundamental invariant is $$K_{ij}(\om^{-1,-1}_{0,-1})+\om^{-1,0}_{0,1}\bullet K_{ij},$$ where $\om^{-1,-1}_{0,-1}$ is conjugated to $\om^{-1,0}_{0,1}$ and $\bullet$ is an action of $\om^{-1,0}_{0,1}$ on $\fg_{0,-1}$ given by the (adjoint) action of a (symmetric) matrix in $\fg_{0,0}$ on $\fg_{0,-1}$.

This is a generic case of $N^2+N-1$--dimensional $2$--nondegenerate CR submanifolds in $\C^{\frac{N(N+1)}{2}}$ with $(N-1)N$--dimensional Levi--kernel.
\end{thm*}
\begin{proof}
As in the proof of Proposition \ref{under}, for every $x$, we consider $\ba_0$ over sufficiently small neighborhood $U$ as an open subset $N\times G_0$, where $N$ is the leaf space with leafs corresponding to integral manifolds of the Levi kernel $\mathcal{K}$. Now, different $G_0$--equivariant extensions of the pseudo one--forms $\theta_i$ to $\fg_{-}$--valued one--forms on $T(N\times G_0)$ such that the property (4) of pseudo--$G_0$--structure holds for the extensions modulo $\fg_{i+j+1}$ define different contact projective geometries, see \cite[Section 4.2]{parabook}. Among them, there is a flat contact projective geometry $(\ba\to N,\tilde \om)$ which is according to Theorem \ref{harm} the only contact projective structure such that the corresponding complex structure $\mathcal{\tilde I}$ on the correspondence space is integrable. So as before, we obtain a section $G$ of $\ba_0\times_{G_{0,I}} \fg_{-1,-1}^*\otimes \fg_{0,1}$ corresponding to $\mathcal{I}-\mathcal{\tilde I}$. 

Since the normalization of the flat contact projective geometry $(\ba\to N,\tilde \om)$ and the absolute parallelism $\om$ solving the equivalence problem of CR structures from \cite[Section 4.2]{Gr19} differ, we measure the difference of the complexifications of the pullbacks $\sigma^*\omega,\sigma^*\tilde \om$ by a Weyl structure $\sigma:\ba_0\to \ba$ of the contact projective geometry using the following formula $$\sigma^*\omega=(\id+\omega^{i,j}_{k,l})\circ (\sigma^*\tilde \om+G),$$ where $\omega^{i,j}_{k,l}$ are sections of $\ba_0\times_{G_{0,I}} \fg_{i,j}^*\otimes \fg_{k,l}$. A priory, $\sigma^*\omega$ has the following block structure
\begin{gather*}
\left[ \begin{array}{cccc}
0&\om^{-2,-1}_{1,0}& \om^{-2,-1}_{1,1}& \om^{-2,-1}_{2,1}\\
0&\om^{-2,-1}_{0,0}&\om^{-2,-1}_{0,1}& *\\
0& \om^{-2,-1}_{0,-1}&* & *\\
1& 0& 0& 0
\end{array}\right]\circ \theta_{-2,-1}  \\
+\left[ \begin{array}{cccc}
\om'{}^{-1,-1}_{0,0}&\om^{-1,-1}_{1,0}& \om^{-1,-1}_{1,1}& \om^{-1,-1}_{2,1}\\
0&\om^{-1,-1}_{0,0}&0& *\\
1& \om^{-1,-1}_{0,-1}&* & *\\
0& 0& 0& *
\end{array}\right]\circ \theta_{-1,-1}\\
+\left[ \begin{array}{cccc}
\om'{}^{-1,0}_{0,0}&\om^{-1,0}_{1,0}& \om^{-1,0}_{1,1}& \om^{-1,0}_{2,1}\\
1&\om^{-1,0}_{0,0}&\om^{-1,0}_{0,1}& *\\
0& 0&* & *\\
0& 0& 0& *
\end{array}\right]\circ \theta_{-1,0}\\
+ \left[ \begin{array}{cccc}
0&\om^{0,-1}_{1,0}& \om^{0,-1}_{1,1}& \om^{0,-1}_{2,1}\\
0&0&0& *\\
0&1&0 & *\\
0& 0& 0& 0
\end{array}\right]\circ(\theta_{0,-1}-\bar G^j_{\bar i\bar k}\theta_{-1,0})\\
+ \left[ \begin{array}{cccc}
0&\om^{0,1}_{1,0}& \om^{0,1}_{1,1}& \om^{0,1}_{2,1}\\
0&0&1& *\\
0& 0&0 & *\\
0& 0& 0& 0
\end{array}\right]\circ(\theta_{0,1}-G^{\bar j}_{ik}\theta_{-1,-1})\\
+\omega_{G_{0,0}},
\end{gather*}
where $\theta_{i,j}$ is the decomposition of the complexification of the pseudo one--forms $\theta_i$ according to the bigrading and $\omega_{G_{0,0}}$ is the complexification of Maurer--Cartan form on $G_{0,I}$. Let us emphasize that the component
\begin{gather*}
\left[ \begin{array}{cccc}
\omega_{G_{0,0}}&0&0&0\\
\theta_{-1,0}&\omega_{G_{0,0}}&\theta_{0,1}& *\\
\theta_{-1,-1}&\theta_{0,-1}&* & *\\
\theta_{-2,-1}& *& *& *
\end{array}\right]\\
\end{gather*}
is the complexification of the Maurer--Cartan form on the Lie group $\exp(\fg_-)\rtimes G_{0,I}\subset \ba_0$ provided by the automorphisms of the flat model $(\ba\to N,\tilde \om)$ on an open subset of $\ba_0$. We evaluate the structure equations on the constant vector fields w.r.t. this Maurer--Cartan form and obtain sections $R_{i,j;k,l}^{m,n}$ of $$\ba_0\times_{G_{0,I}} \fg_{i,j}^*\otimes \fg_{k,l}^*\otimes \fg_{m,n}.$$

We proceed by investigating the normalization conditions on $R_{i,j;k,l}^{m,n}$  according to their homogeneity w.r.t. the grading given by $\Sigma_1=\{1\}$.

In homogeneity $1$, we can use the vanishing of $R_{-2,-1;0,1}^{-1,-1},R_{-2,-1;-1,-1}^{-2,-1},R_{-2,-1;0,1}^{-1,0}$ and their conjugates to normalize $$\om^{0,-1}_{1,0}=0, \om^{0,-1}_{1,1}=0, \om^{0,1}_{1,0}=0,  \om^{0,1}_{1,1}=0, \om'{}^{-1,0}_{0,0}=0, \om'{}^{-1,-1}_{0,0}=0.$$ Further, we can use vanishing of $R_{-1,0;-1,-1}^{-1,-1}$ and their conjugates to normalize $$\om^{-1,0}_{0,0}=-\om^{-1,-1}_{0,-1}, \om^{-1,-1}_{0,0}=\om^{-1,0}_{0,1}.$$ Thus it remains to normalize $\om^{-1,0}_{0,1}$ and $\om^{-1,-1}_{0,-1}$ that are conjugated.

What remains in homogeneity $1$ is $$R_{-1,-1;0,1}^{0,1}=d_{0,1}G+\om^{-1,0}_{0,1}([\theta_{-1,-1},\theta_{0,1}])+[\om^{-1,-1}_{0,0}\circ \theta_{-1,-1},\theta_{0,1}]$$ and $$R_{-1,-1;0,-1}^{0,-1}=d_{0,-1}\om^{-1,-1}_{0,-1}+[\om^{-1,-1}_{0,0}\circ \theta_{-1,-1},\theta_{0,-1}]$$ and their conjugates. We can decompose $d_{0,1}G$ to highest weight component, which provides the first fundamental invariant and we can normalize to $0$ the remaining components using $$\om^{-1,0}_{0,1}.$$ The computation shows that the remaining component vanishes. Therefore only $R_{-1,-1;0,-1}^{0,-1}$ provides the second invariant.

So it remains to show that there are no further fundamental invariants, i.e., to show that, if we assume that these invariants vanish, then $\om$ is flat Cartan connection. We can compute that there is no nonzero $R_{i,j;k,l}^{m,n}$  in homogeneity $1$ under this assumption. We can not present the further computations in detail, because there are too many differential consequences of the integrability conditions and the assumption of the vanishing of the above invariant. We used the computer algebra program Maple with its DifferentialGeometry package \cite{dgpack} for these computations.

In homogeneity $2$,  we can use the vanishing of $R'{}_{-2,-1;0,1}^{0,0},R_{-2,-1;-1,0}^{-1,-1},R_{-2,-1;-1,-1}^{-1,-1}$ and their conjugates to normalize \begin{gather*}\om^{0,1}_{2,1}=0,  \om^{0,-1}_{2,1}=0,\\
 \om^{-1,-1}_{1,0}=\om^{-2,-1}_{0,0}, \om^{-1,-1}_{1,1}=\om^{-2,-1}_{0,1}, \om^{-1,0}_{1,0}=-\om^{-2,-1}_{0,-1}, \om^{-1,0}_{1,1}=(\om^{-2,-1}_{0,0})^T.\end{gather*} Then we can use vanishing of $R_{-1,-1;-1,-1}^{0,0}$ and its conjugate to determine $$\om^{-2,-1}_{0,-1},\om^{-2,-1}_{0,1}$$ and we can use vanishing of $R_{-1,-1;0,1}^{1,0},R_{-1,-1;0,1}^{1,1},R_{-2,-1;0,1}^{0,1}$ and their conjugates to determine $$\om^{-2,-1}_{0,0}.$$ After this, one can check that all other components of $R_{i,j;k,l}^{m,n}$  in homogeneity $2$ vanish under our assumptions.

In homogeneity $3$,  we can use the vanishing of $R'{}_{-2,-1;-1,-1}^{0,0}$ and its conjugate to normalize  $$\om^{-1,0}_{2,1}=\om^{-2,-1}_{1,0}, \om^{-1,-1}_{2,1}=\om^{-2,-1}_{1,1}.$$ Then we can use vanishing of $R_{-1,-1;-1,0}^{1,0}$ and their conjugates to determine $$\om^{-2,-1}_{1,0},\om^{-2,-1}_{1,1}.$$ After this, one can check that all other components of $R_{i,j;k,l}^{m,n}$  in homogeneity $3$ vanish under our assumptions.

In  homogeneity $4$, we can use vanishing of $R_{-1,-1;-1,0}^{2,1}$ to determine $$\om^{-2,-1}_{2,1}.$$ This completely determines $\si^*\om$ and  one can check that all $R_{i,j;k,l}^{m,n}$ vanish, i.e., highest weight component of $R_{-1,-1;0,1}^{0,1}$ and $R_{-1,-1;0,-1}^{0,-1}$ are the fundamental invariants of these CR geometries.

The genericity follows from the discussion in the proof of Theorem \ref{fin1}.
\end{proof}

Let us provide an example for the case $N=3$.  We start with the flat model, for which we can obtain the following defining equation using \cite[Proposition 3.7]{Gr19}:

\begin{align*}
\Re(w)&=\frac{ 1}
 {P}(-2{ 
\bar z_1} { z_1}-2{ 
\bar z_2} { z_2}+\left( { \bar z_5} { z_3} { z_5}-{ \bar z_5} {{ z_4}}^{2}-{ 
z_3} \right) {{ \bar z_1}}^{2}+ \left( { \bar z_3} { \bar z_5} { z_5}-{{ \bar z_4}}^{2}{
 z_5}-{ \bar z_3} \right) {{ z_1}}^{2}\\
& -2\left( { \bar z_4} { z_3} { z_5}- { \bar z_4} {{ z_4}}^{2}+ { z_4} \right) { \bar z_1} { \bar z_2}-2\left(  { \bar z_3} { 
\bar z_5} { z_4}- {{ \bar z_4}}^{2}{ z_4}+ { \bar z_4} \right) { z_1
} { z_2}\\
&+2 \left(  { \bar z_4} { z_4}+ { \bar z_5} { z_5} \right) { 
\bar z_1} { z_1}
-2 \left( { \bar z_4} { z_3}+ { \bar z_5} { z_4}
 \right) { \bar z_1} { z_2}
-2 \left( 
{ \bar z_3} { z_4}+ { \bar z_4} { z_5} \right) { \bar z_2} { z_1}
\\
&+2 \left( { \bar z_3} { z_3}+{ \bar z_4} { z_4} \right) { 
\bar z_2} { z_2}+ \left( { \bar z_3} { \bar z_5} { z_3}-{{ \bar z_4}}^{2}{
 z_3}-{ \bar z_5} \right) {{ z_2}}^{2} + \left( { \bar z_3} { z_3} { z_5}-{
 \bar z_3} {{ z_4}}^{2}-{ z_5} \right) {{ \bar z_2}}^{2}),\\
 P&:={ \bar z_3} { \bar z_5} { z_3} { z_5}-{ \bar z_3} { \bar z_5} {{ 
z_4}}^{2}-{{ \bar z_4}}^{2}{ z_3} { z_5}+{{ \bar z_4}}^{2}{{ z_4}}^{
2}-{ \bar z_3} { z_3}-2 { \bar z_4} { z_4}-{ \bar z_5} { z_5}+1,
\end{align*}

We pick the following complex frame of $\mathcal{D}/\mathcal{K}:$

\begin{align*}
\mathcal{L}_1&:=\frac{\partial}{\partial z_1}\sqrt{P}-i\frac{\partial}{\partial \Im(w)}\frac{2Q}{\sqrt{P}}\\
\mathcal{L}_2&:=\frac{\partial}{\partial z_1}T+\frac{\partial}{\partial z_2}S-i\frac{\partial}{\partial \Im(w)}2U\\
 Q&:={ \bar z_3}\,{ \bar z_5}\,{ z_1}\,{ z_5}-{ \bar z_3}\,{ \bar z_5}\,{ z_2
}\,{ z_4}-{{ \bar z_4}}^{2}{ z_1}\,{ z_5}+{{ \bar z_4}}^{2}{ z_2}
\,{ z_4}+{ \bar z_1}\,{ \bar z_4}\,{ z_4}+{ \bar z_1}\,{ \bar z_5}\,{ 
z_5}\\
&-{ \bar z_2}\,{ \bar z_3}\,{ z_4}-{ \bar z_2}\,{ \bar z_4}\,{ z_5}-{
 z_1}\,{ \bar z_3}-{ \bar z_4}\,{ z_2}-{ \bar z_1}\\
 S&:=\bar z_4z_4+\bar z_5z_5-1\\
 T&:=\bar z_4z_3+\bar z_5z_4\\
 U&:=\bar z_4z_1+\bar z_5z_2+\bar z_2
\end{align*}
and the following complex frame of $\mathcal{K}:$

\begin{align*}
\mathcal{K}^{\bar 1}_1&:=-\frac{\partial}{\partial z_1}\frac{Q}{S}-\frac{\partial}{\partial z_3}\frac{P}{S}+i\frac{\partial}{\partial \Im(w)}\frac{Q^2}{SP}\\
\mathcal{K}^{\bar 1}_2&+\mathcal{K}^{\bar 2}_1:=-\frac{\partial}{\partial z_1}(\frac{2TQ}{\sqrt{P}S}-\frac{Q'}{\sqrt{P}})-\frac{\partial}{\partial z_2}\frac{Q}{\sqrt{P}}-\frac{\partial}{\partial z_3}\frac{2T\sqrt{P}}{S}-\frac{\partial}{\partial z_4}\sqrt{P}\\
&+i\frac{\partial}{\partial \Im(w)}(\frac{2Q^2T}{P\sqrt{P}S}+\frac{2QQ'}{P\sqrt{P}})\\
\mathcal{K}^{\bar 2}_2&:=-\frac{\partial}{\partial z_1}(\frac{QT^2}{PS}-\frac{QT}{P})-\frac{\partial}{\partial z_2}(\frac{QT}{P}-\frac{Q'S}{P})-\frac{\partial}{\partial z_3}\frac{T^2}{S}-\frac{\partial}{\partial z_4}T-\frac{\partial}{\partial z_5}S\\
&+i\frac{\partial}{\partial \Im(w)}(\frac{Q^2T^2}{P^2S}+\frac{2QQ'T}{P^2}-\frac{Q'^2S}{P^2})\\
&Q':=\bar z_3\bar z_5z_1z_4-\bar z_3\bar z_5z_2z_3-\bar z_4^2z_1z_4+\bar z_4^2z_2z_3+\bar z_1\bar z_4z_3+\bar z_1\bar z_5z_4-\bar z_2\bar z_3z_3-\bar z_2\bar z_4z_4\\
&+\bar z_4z_1+\bar z_5z_2+\bar z_2
\end{align*}

It is not hard to verify that $\mathcal{L}_1,\mathcal{L}_2$ and $\mathcal{\bar L}_1, \mathcal{\bar L}_2$ generate graded Lie algebra isomorphic to $\fg_-$ of a contact distribution and that $\mathcal{K}^{\bar 1}_1,\mathcal{K}^{\bar 1}_2+\mathcal{K}^{\bar 2}_1,\mathcal{K}^{\bar 2}_2$ map appropriate $\mathcal{\bar L}_{\bar i}$ on $\mathcal{L}_j$ (modulo Levi kernel).

We construct complexification of the pullback $s^*\si^*\om$, where $s: M\to \ba_0$ is the (local) section provided by our choice of the frame. We use the notation of the proof of Theorem \ref{fin2} and assume that $s^*\theta_{-2,-1},s^*\theta_{-1,-1},s^*\theta_{0,-1}$ is the dual coframe to our frame. As in \cite[Section 4.1]{Gr19}, we firstly (re)construct the infinitesimal $\fg_0$--structure, i.e., determine the image of $s^*\theta_{0,-1}$ in $\fg_{0,0}$ by vanishing of the components $$\mathcal{R}_{i,j;k,l}^{m,n}:=s^*R_{i,j;k,l}^{m,n}$$ of homogeneity $0,1$ w.r.t. the bigrading
\begin{gather*}
\left[ \begin{array}{cccc}
W&0&0&0\\
s^*\theta_{-1,0}&V^{\bar i}_{\bar j}&s^*\theta_{0,1}& *\\
s^*\theta_{-1,-1}&s^*\theta_{0,-1}&* & *\\
s^*\theta_{-2,-1}& *& *& *
\end{array}\right],\\
W:=\frac{\sqrt{P}\bar z_4}{2S}(k^{\bar 1}_2+k^{\bar 2}_1)+\frac{\sqrt{P}z_4}{2S}(\bar k^{1}_{\bar  2}+\bar k^{2}_{\bar 1})+\frac{\bar z_4T+\bar z_5S}{2S}k^{\bar 2}_2+\frac{z_4\bar T+z_5S}{2S}\bar k^2_{\bar 2}\\
V^{\bar 1}_{\bar 1}=\frac{\bar z_4\bar T-\bar z_3S}{2S}k^{\bar 1}_1+\frac{-z_4 T+z_3S}{2S}\bar k^1_{\bar 1}-\frac{\sqrt{P}\bar z_4}{2S}(k^{\bar 1}_2+k^{\bar 2}_1)+\frac{\sqrt{P}z_4}{2S}(\bar k^{1}_{\bar  2}+\bar k^{2}_{\bar 1})\\
V^{\bar 1}_{\bar 2}=\frac{\bar z_4\bar T-\bar z_3S}{2S}(k^{\bar 1}_2+k^{\bar 2}_1)-\frac{\sqrt{P}\bar z_4}{2S}k^{\bar 2}_2\\
V^{\bar 2}_{\bar 1}=-\bar V^{\bar 1}_{\bar 2}\\
V^{\bar 2}_{\bar 2}=-\frac{\sqrt{P}\bar z_4}{2S}(k^{\bar 1}_2+k^{\bar 2}_1)+\frac{\sqrt{P}z_4}{2S}(\bar k^{1}_{\bar  2}+\bar k^{2}_{\bar 1})-\frac{\bar z_4T+\bar z_5S}{2S}k^{\bar 2}_2+\frac{z_4\bar T+z_5S}{2S}\bar k^2_{\bar 2}
\end{gather*}
where $k^{\bar i}_j$ are dual to $K^{\bar i}_j$.

Now, we change the complex structure by taking $$\mathcal{L}_i+G^{\bar j}_{ik}\mathcal{\bar K}^k_{\bar j}$$ as a complex frame of $\mathcal{D}/\mathcal{K}$. The integrability conditions are quite complicated in this case and the PDE solver in Maple software did not provide us their general solution. Assuming, $G^{\bar j}_{ik}=0$ except $G^{\bar 1}_{11}$ and assuming that $G^{\bar 1}_{11}$ does not depend on $w$ provides two sets of solutions of integrability conditions. We present here the simpler one that depends on one complex function $F$ of four variables:
$$
G^{\bar 1}_{11}=\frac{F(\bar z_1,\bar z_3,\bar z_4,\bar z_5)S^2}{\sqrt{P^3}}.
$$
We normalize the components $\mathcal{R}_{i,j;k,l}^{m,n}$ in homogeneity $1$ w.r.t. the grading given by $\Sigma_1$ as in Theorem \ref{fin2} and obtain
\begin{gather*}
s^*\theta_{-1,-1}\circ \left[ \begin{array}{cccc}
0&?&?&?\\
0&\om^{-1,0}_{0,1}&-G_{ijk}& *\\
1& \om^{-1,-1}_{0,-1}&* & *\\
0& 0& 0& *
\end{array}\right]+s^*\theta_{-1,0}\circ \left[ \begin{array}{cccc}
0&?&?&?\\
1&- \om^{-1,-1}_{0,-1}&\om^{-1,0}_{0,1}& *\\
0& -\bar G_{ijk}&* & *\\
0& 0& 0& *
\end{array}\right]\\
+s^*\theta_{-2,-1}\circ \left[ \begin{array}{cccc}
0&?&?&?\\
0&?&?& *\\
0&?&* & *\\
1& 0& 0& 0
\end{array}\right]+\left[ \begin{array}{cccc}
W&0&0&0\\
0&V^{\bar i}_{\bar j}&s^*\theta_{0,1}& 0\\
0&s^*\theta_{0,-1}&* & 0\\
0&0& 0& *
\end{array}\right],\\
\end{gather*}
where $?$ are components of higher homogeneity we do not need to compute,
\begin{align*}
(\om^{-1,0}_{0,1})_{11}^{1}&=-\bar F\frac{S(\bar z_4\bar T-\bar z_3S)}{\sqrt{P^3}}+\bar F_{z_3}\frac{3S}{10\sqrt{P}}+\bar F_{z_1}\frac{3SQ}{10\sqrt{P^3}}\\
(\om^{-1,0}_{0,1})_{21}^{1}&=-\bar F\frac{S\bar z_4}{P}+\bar F_{z_3}\frac{2ST}{5P}+\bar F_{z_4}\frac{S^2}{5P}+\bar F_{z_1}\frac{S(2QT-SQ')}{5P^2}\\
(\om^{-1,0}_{0,1})_{12}^{1}&=0\\
(\om^{-1,0}_{0,1})_{22}^{1}&=-\bar F_{z_3}\frac{S}{20\sqrt{P}}-\bar F_{z_1}\frac{SQ}{20\sqrt{P^3}}\\
(\om^{-1,0}_{0,1})_{12}^{2}&=0\\
(\om^{-1,0}_{0,1})_{22}^{2}&=0,
\end{align*}
and $\om^{-1,-1}_{0,-1}$ is conjugated to $\om^{-1,0}_{0,1}$.

We compute that there are the following nontrivial fundamental invariants
\begin{align*}
(\mathcal{R}_{-1,-1;0,1}^{0,1})_{11}^1{}^1_1&=-\frac12(\mathcal{R}_{-1,-1;0,1}^{0,1})_{21}^2{}^1_1=-F_{\bar z_1}\frac{\bar QS}{10\sqrt{P^3}}-F_{\bar z_3}\frac{S}{10\sqrt{P}}\\
(\mathcal{R}_{-1,-1;0,1}^{0,1})_{11}^2{}^1_1&=-2(\mathcal{R}_{-1,-1;0,1}^{0,1})_{22}^2{}^1_1=- F_{\bar z_3}\frac{4S\bar T}{5P}-F_{\bar z_4}\frac{2S^2}{5P}-2F_{\bar z_1}\frac{S(2\bar Q\bar T-S\bar Q')}{5P^2}\\
(\mathcal{R}_{-1,-1;0,1}^{0,1})_{11}^2{}^1_2&=-2(\mathcal{R}_{-1,-1;0,1}^{0,1})_{22}^2{}^1_2=F_{\bar z_1}\frac{\bar QS}{5\sqrt{P^3}}+F_{\bar z_3}\frac{S}{5\sqrt{P}}\\
(\mathcal{R}_{-1,-1;0,1}^{0,1})_{12}^2{}^1_1&=-F_{\bar z_1}\frac{S\bar T\bar U}{\sqrt{P^3}}-F_{\bar z_5}\frac{S^3}{\sqrt{P^3}}-F_{\bar z_4}\frac{S^2\bar T}{\sqrt{P^3}}-F_{\bar z_3}\frac{S\bar T^2}{\sqrt{P^3}}\\
 \end{align*}
 \begin{align*}
(\mathcal{R}_{-1,-1;0,-1}^{0,-1})_{11}^1{}^1_1&=F\frac{2S(-z_4 T+z_3S)}{\sqrt{P^3}}+\bar F\frac{2(-\bar z_4 \bar T+\bar z_3S)^2}{\sqrt{P^3}}-F_{\bar z_1}\frac{3\bar QS}{5\sqrt{P^3}}-F_{\bar z_3}\frac{3S}{5\sqrt{P}}\\
&-\bar F_{z_1}\frac{13Q(-\bar z_4 \bar T+\bar z_3S)}{10\sqrt{P^3}}-\bar F_{z_3}\frac{13(-\bar z_4 \bar T+\bar z_3S)}{10\sqrt{P}}+\bar F_{z_1z_1}\frac{3Q^2}{10\sqrt{P^3}}\\
&+\bar F_{z_1z_3}\frac{3Q}{5\sqrt{P}}+\bar F_{z_3z_3}\frac{3\sqrt{P}}{10}\\
(\mathcal{R}_{-1,-1;0,-1}^{0,-1})_{21}^1{}^1_1&=\bar F\frac{2(-\bar z_4 \bar T+\bar z_3S)\bar z_4}{P}-\bar F_{z_1}\frac{(7P\bar z_4+4S(\bar z_5\bar T-\bar z_4S))Q-2Q'(-\bar z_4 \bar T+\bar z_3S)S}{5P^2}\\
&-\bar F_{z_3}\frac{(7P\bar z_4+4S(\bar z_5\bar T-\bar z_4S))}{5P^2}-\bar F_{z_4}\frac{2S(-\bar z_4 \bar T+\bar z_3S)}{5P}\\
&+\bar F_{z_1z_1}\frac{Q(QS+2PU)}{5P^2}+\bar F_{z_1z_3}\frac{Q(QS+2PU)}{5P}+\bar F_{z_3z_3}\frac{2T}{5}\\
&+\bar F_{z_1z_4}\frac{QS}{5P}+\bar F_{z_3z_4}\frac{S}{5}\\
(\mathcal{R}_{-1,-1;0,-1}^{0,-1})_{11}^1{}^1_2&=F\frac{z_4S}{P}-F_{\bar z_1}\frac{S(2\bar Q\bar T-S\bar Q')}{5P^2}-F_{\bar z_3}\frac{2S \bar T}{5P}-F_{\bar z_4}\frac{S^2}{5P}\\
(\mathcal{R}_{-1,-1;0,-1}^{0,-1})_{21}^1{}^1_2&=F_{\bar z_1}\frac{\bar QS}{20\sqrt{P^3}}+F_{\bar z_3}\frac{S}{20\sqrt{P}}+\bar F_{z_1}\frac{Q(-\bar z_4 \bar T+\bar z_3S)}{20\sqrt{P^3}}+\bar F_{z_3}\frac{(-\bar z_4 \bar T+\bar z_3S)}{20\sqrt{P}}\\
&-\bar F_{z_1z_1}\frac{Q^2}{20\sqrt{P^3}}-\bar F_{z_1z_3}\frac{Q}{10\sqrt{P}}-\bar F_{z_3z_3}\frac{\sqrt{P}}{20}\\
(\mathcal{R}_{-1,-1;0,-1}^{0,-1})_{11}^2{}^1_1&=\bar F\frac{2(-\bar z_4 \bar T+\bar z_3S)\bar z_4}{P}-\bar F_{z_1}\frac{(7P\bar z_4+4S(\bar z_5\bar T-\bar z_4S))Q-2Q'(-\bar z_4 \bar T+\bar z_3S)S}{10P^2}\\
&-\bar F_{z_3}\frac{(7P\bar z_4+4S(\bar z_5\bar T-\bar z_4S))}{10P^2}-\bar F_{z_4}\frac{4S(-\bar z_4 \bar T+\bar z_3S)}{5P}\\
&+\bar F_{z_1z_1}\frac{3Q(QS+2PU)}{10P^2}+\bar F_{z_1z_3}\frac{3Q(QS+2PU)}{10P}+\bar F_{z_3z_3}\frac{3T}{5}\\
&+\bar F_{z_1z_4}\frac{3QS}{10P}+\bar F_{z_3z_4}\frac{3T}{10}\\
(\mathcal{R}_{-1,-1;0,-1}^{0,-1})_{21}^2{}^1_1&=\bar F\frac{2\bar z_4^2}{\sqrt{P}}-\bar F_{z_1}\frac{2(3\bar z_4SQ'+Q(5\bar z_4T-\bar z_5S))}{5\sqrt{P^3}}-\bar F_{z_3}\frac{2(5\bar z_4T-\bar z_5S)}{5\sqrt{P}}-\bar F_{z_4}\frac{6\bar z_4S}{5\sqrt{P}}\\
&+\bar F_{z_1z_1}\frac{(2TQ-SQ')^2}{5\sqrt{P^5}}+\bar F_{z_1z_3}\frac{4(2TQ-SQ')T}{5\sqrt{P^3}}+\bar F_{z_3z_3}\frac{4T^2}{5\sqrt{P}}\\
&+\bar F_{z_1z_4}\frac{2S(2TQ-SQ')}{5\sqrt{P^3}}+\bar F_{z_3z_4}\frac{4TS}{5\sqrt{P}}+\bar F_{z_4z_4}\frac{S^2}{5\sqrt{P}}\\
(\mathcal{R}_{-1,-1;0,-1}^{0,-1})_{11}^2{}^1_2&=F\frac{S(-z_4 T+z_3S)}{\sqrt{P^3}}+\bar F\frac{S^2}{\sqrt{P^3}}-F_{\bar z_1}\frac{\bar QS}{4\sqrt{P^3}}-F_{\bar z_3}\frac{S}{4\sqrt{P}}\\
&-\bar F_{z_1}\frac{7Q(-\bar z_4 \bar T+\bar z_3S)}{20\sqrt{P^3}}-\bar F_{z_3}\frac{7(-\bar z_4 \bar T+\bar z_3S)}{20\sqrt{P}}\\
(\mathcal{R}_{-1,-1;0,-1}^{0,-1})_{21}^2{}^1_2&=\bar F\frac{(-\bar z_4 \bar T+\bar z_3S)\bar z_4}{P}-\bar F_{z_1}\frac{(7P\bar z_4+4S(\bar z_5\bar T-\bar z_4S))Q-2Q'(-\bar z_4 \bar T+\bar z_3S)S}{20P^2}\\
&-\bar F_{z_3}\frac{(7P\bar z_4+4S(\bar z_5\bar T-\bar z_4S))}{20P^2}-\bar F_{z_4}\frac{2S(-\bar z_4 \bar T+\bar z_3S)}{5P}\\
&-\bar F_{z_1z_1}\frac{Q(QS+2PU)}{20P^2}-\bar F_{z_1z_3}\frac{Q(QS+2PU)}{20P}-\bar F_{z_3z_3}\frac{2T}{10}\\
&-\bar F_{z_1z_4}\frac{QS}{20P}-\bar F_{z_3z_4}\frac{S}{20}\\
 \end{align*}
  \begin{align*}
(\mathcal{R}_{-1,-1;0,-1}^{0,-1})_{11}^2{}^2_2&=F\frac{2z_4S}{P}-F_{\bar z_1}\frac{S(2\bar Q\bar T-S\bar Q')}{5P^2}-F_{\bar z_3}\frac{4S\bar T}{5P}-F_{\bar z_4}\frac{2S^2}{5P}\\
(\mathcal{R}_{-1,-1;0,-1}^{0,-1})_{21}^2{}^2_2&=F_{\bar z_1}\frac{\bar QS}{10\sqrt{P^3}}+F_{\bar z_3}\frac{S}{10\sqrt{P}}+\bar F_{z_1}\frac{Q(-\bar z_4 \bar T+\bar z_3S)}{10\sqrt{P^3}}+\bar F_{z_3}\frac{(-\bar z_4 \bar T+\bar z_3S)}{10\sqrt{P}}\\
(\mathcal{R}_{-1,-1;0,-1}^{0,-1})_{12}^2{}^1_1&=\bar F_{z_1}\frac{-\bar z_4SQ'+T(2\bar z_4Q-5U(\bar z_3S-\bar z_4\bar T))}{5\sqrt{P^3}}\\
&-\bar F_{z_3}\frac{T(5T(-\bar z_4 \bar T+\bar z_3S)-2\bar z_4P)}{5\sqrt{P^3}}-\bar F_{z_4}\frac{S(5T(-\bar z_4 \bar T+\bar z_3S)-\bar z_4P)}{5\sqrt{P^3}}\\
&-\bar F_{z_5}\frac{S^2(-\bar z_4 \bar T+\bar z_3S)}{\sqrt{P^3}}+\bar F_{z_1z_1}\frac{3TQU}{10\sqrt{P^3}}+\bar F_{z_1z_3}\frac{3T(QT+PU)}{10\sqrt{P^3}}\\
&+\bar F_{z_3z_3}\frac{3T^2}{10\sqrt{P}}+\bar F_{z_3z_5}\frac{3S^2}{10\sqrt{P}}+\bar F_{z_1z_4}\frac{3TQS}{10\sqrt{P^3}}+\bar F_{z_1z_5}\frac{3QS^2}{10\sqrt{P^3}}\\
&+\bar F_{z_3z_4}\frac{3ST}{10\sqrt{P}}\\
(\mathcal{R}_{-1,-1;0,-1}^{0,-1})_{22}^2{}^1_1&=\bar F_{z_1}\frac{SUP\bar z_4+T(4SQ'\bar z_4-(3\bar z_4T-\bar z_5S)Q)}{5P^2}-\bar F_{z_3}\frac{T(3\bar z_4T-2\bar z_5S)}{5P}\\
&-\bar F_{z_4}\frac{S(4\bar z_4T-\bar z_5 S)}{5P}-\bar F_{z_5}\frac{S^2\bar z_4}{P}+\bar F_{z_1z_1}\frac{TU(QT+UP)}{5P^2}\\
&+\bar F_{z_1z_3}\frac{T^2(QT+3PU)}{5P^2}+\bar F_{z_3z_3}\frac{2T^3}{5P}+\bar F_{z_3z_5}\frac{2S^2T}{5P}+\bar F_{z_1z_4}\frac{ST(QT+2UP)}{5P^2}\\
&+\bar F_{z_1z_5}\frac{S^2(QT+UP)}{5P^2}+\bar F_{z_3z_4}\frac{3ST^2}{5P}+\bar F_{z_4z_4}\frac{S^2T}{5P}+\bar F_{z_4z_5}\frac{S^3}{5P}\\
(\mathcal{R}_{-1,-1;0,-1}^{0,-1})_{12}^2{}^1_2&=\bar F\frac{\bar z_4(-\bar z_4 \bar T+\bar z_3S)}{P}-\bar F_{z_1}\frac{7Q\bar z_4}{20P}-\bar F_{z_3}\frac{7\bar z_4}{20}\\
(\mathcal{R}_{-1,-1;0,-1}^{0,-1})_{22}^2{}^1_2&=\bar F\frac{\bar z_4^2}{\sqrt{P}}-\bar F_{z_1}\frac{\bar z_4(QT+UP)}{5\sqrt{P^3}}-\bar F_{z_3}\frac{2T\bar z_4}{5\sqrt{P}}-\bar F_{z_4}\frac{\bar z_4S}{5\sqrt{P}}\\
&-\bar F_{z_1z_1}\frac{TQU}{20\sqrt{P^3}}-\bar F_{z_1z_3}\frac{(QT+UP)T}{20\sqrt{P^3}}-\bar F_{z_3z_3}\frac{T^2}{20\sqrt{P}}\\
&-\bar F_{z_3z_5}\frac{S^2}{20\sqrt{P}}-\bar F_{z_1z_4}\frac{TQS}{20\sqrt{P^3}}-\bar F_{z_3z_4}\frac{TS}{20\sqrt{P}}-\bar F_{z_1z_5}\frac{QS^2}{20\sqrt{P^3}}\\
(\mathcal{R}_{-1,-1;0,-1}^{0,-1})_{12}^2{}^2_2&=F_{\bar z_1}\frac{\bar QS}{10\sqrt{P^3}}+F_{\bar z_3}\frac{S}{10\sqrt{P}}\\
(\mathcal{R}_{-1,-1;0,-1}^{0,-1})_{22}^2{}^2_2&=\bar F_{z_1}\frac{Q\bar z_4}{10P}+\bar F_{z_3}\frac{\bar z_4}{10}\\
 \end{align*}
 
We can check that these CR geometries are flat if and only if $F=0$.
 
\appendix
 
\section{Five--dimensional uniformly 2-nondegenerate submanifolds in $\C^3$}

The case $\frak{sp}(4,\mathbb{R}),\Sigma_1=\{1\}$ is well-studied and it is known that there are two fundamental invariants $\mathcal{W}$ and $\mathcal{J}$, see \cite{Poc13,Gr19}. The results in previous Sections allow us to construct (locally) all corresponding CR geometries.

We use the defining equation

\begin{align*}
\Re(w)&=\frac { { \bar z_2}\,{{ z_1}}^{2}+{{ \bar z_1}}^{2}{ z_2
}+2\,{ \bar z_1}\,{ z_1} }{2( \bar z_2z_2-1)}\\
\end{align*}

and we pick the following complex frame of $\mathcal{D}/\mathcal{K}$:

\begin{align*}
\mathcal{L}_1&:=\frac{\partial}{\partial z_1}+i\frac{\partial}{\partial \Im(w)}\frac{z_1\bar z_2+\bar z_1}{\bar z_2z_2-1}
\end{align*}

and the following complex frame of $\mathcal{K}$:

\begin{align*}
\mathcal{K}^{\bar 1}_1&:=-\frac{\partial}{\partial z_1}(z_1\bar z_2+\bar z_1)-\frac{\partial}{\partial z_2}(\bar z_2z_2-1)-i\frac{\partial}{\partial \Im(w)}\frac{(z_1\bar z_2+\bar z_1)^2}{2(\bar z_2z_2-1)}
\end{align*}

It is not hard to verify that $L_1$ and $\bar L_1$ generate graded Lie algebra isomorphic to $\fg_-$ of a contact distribution and that $K^1_1$ maps $\bar L_1$ on $L_1$ (modulo Levi kernel).

We construct complexification of the pullback $s^*\si^*\om$, where $s: M\to \ba_0$ is the (local) section provided by our choice of the frame. We use the notation from the proof of Theorem \ref{fin2} and assume that $s^*\theta_{-2,-1},s^*\theta_{-1,-1},s^*\theta_{0,-1}$ is the dual coframe to our frame. As in \cite[Section 4.1]{Gr19}, we firstly construct the infinitesimal $\fg_0$--structure, i.e., determine the image of $s^*\theta_{0,-1}$ in $\fg_{0,0}$ by vanishing of the components $s^*R_{i,j;k,l}^{m,n}$ of homogeneity $0,1$ w.r.t. the bigrading
\begin{gather*}
\left[ \begin{array}{cccc}
W&0&0&0\\
s^*\theta_{-1,0}&V&s^*\theta_{0,1}& *\\
s^*\theta_{-1,-1}&s^*\theta_{0,-1}&* & *\\
s^*\theta_{-2,-1}& *& *& *
\end{array}\right],\\
W:=-\frac{\bar z_2}{2}k^{\bar 1}_1-\frac{z_2}{2}\bar k^1_{\bar 1}\\
V:=-\frac{\bar z_2}{2}k^{\bar 1}_1+\frac{z_2}{2}\bar k^1_{\bar 1}\\
\end{gather*}
where $k^{\bar 1}_1$ is dual to $\mathcal{K}^{\bar 1}_1$.

Now, we change the complex structure by taking $$\mathcal{L}_1+G^{\bar 1}_{11}\mathcal{\bar K}^1_{\bar 1}$$ as a complex frame of $\mathcal{D}/\mathcal{K}$. The integrability of CR geometry is given by the equation $$K^{\bar 1}_1(G^{\bar 1}_{11})=\mathcal{K}^{\bar 1}_1(G^{\bar 1}_{11})-2\bar z_2G^{\bar 1}_{11}=0$$ for which the solver of differential equations in Maple software provides general solution
$$G^{\bar 1}_{11}=\frac{F(\bar z_1,\bar z_2,\frac{\bar z_2 z_1+\bar z_1}{\bar z_2 z_2-1},i\frac{z_1(\bar z_2 z_1+\bar z_1)}{\bar z_2 z_2-1}-2v)}{(\bar z_2 z_2-1)^2}$$
depending on one function of four variables.

We know from \cite[Section 4.3]{Gr19} that we need to normalize the components $$\mathcal{R}_{i,j;k,l}^{m,n}:=s^*R_{i,j;k,l}^{m,n}$$ in homogeneity $1$ and $2$ w.r.t. the grading given by $\Sigma_1$ in order to obtain the fundamental invariants $$\mathcal{W}=-\frac{1}{3}\mathcal{R}_{-1,-1;0,-1}^{0,-1}$$ and $$\mathcal{J}=\frac{1}{12}\mathcal{R}_{-2,-1;-1,-1}^{0,1}.$$ Following the normalization in \cite[Section 4.3]{Gr19}, we obtain

\begin{gather*}
s^*\theta_{-1,-1}\circ \left[ \begin{array}{cccc}
0&\om^{-2,-1}_{0,0}& \om^{-1,-1}_{1,1}&?\\
0&\om^{-1,0}_{0,1}&-G^{\bar 1}_{11}& *\\
1& \om^{-1,-1}_{0,-1}&* & *\\
0& 0& 0& *
\end{array}\right]+s^*\theta_{-1,0}\circ \left[ \begin{array}{cccc}
0&\om^{-1,0}_{1,0}& \om^{-2,-1}_{0,0}&?\\
1&- \om^{-1,-1}_{0,-1}&\om^{-1,0}_{0,1}& *\\
0& -\bar G^1_{\bar 1 \bar 1}&* & *\\
0& 0& 0& *
\end{array}\right]\\
+s^*\theta_{-2,-1}\circ \left[ \begin{array}{cccc}
0&?&?&?\\
0&\om^{-2,-1}_{0,0}&-\om^{-1,-1}_{1,1}& *\\
0&-\om^{-1,0}_{1,0}&* & *\\
1& 0& 0& 0
\end{array}\right]+\left[ \begin{array}{cccc}
W&0&0&0\\
0&V&s^*\theta_{0,1}& 0\\
0&s^*\theta_{0,-1}&* & 0\\
0&0& 0& *
\end{array}\right],\\
\end{gather*}
where $?$ are components in higher homogeneity that we do not need to determine, $\om^{-1,0}_{0,1},\om^{-1,0}_{1,0}$ are conjugated to $\om^{-1,-1}_{0,-1}, \om^{-1,-1}_{1,1}$ and
\begin{align*}
\om^{-1,-1}_{0,-1}&=-\frac{K^{\bar 1}_1(\bar G^{1}_{\bar 1\bar 1})}{3},\\
\om^{-1,-1}_{1,1}&=G^{\bar 1}_{11}\frac{\bar K^1_{\bar 1}(\bar K^1_{\bar 1}(G^{\bar 1}_{11}))}{12}-\frac{\bar K^1_{\bar 1}(G^{\bar 1}_{11})^2}{18}-\frac{\mathcal{\bar L}_{\bar 1}(G^{\bar 1}_{11})}{4}+\frac{\mathcal{L}_1(\bar K^1_{\bar 1}(G^{\bar 1}_{11}))}{12}\\
\om^{-2,-1}_{0,0}&=-\frac{G^{\bar 1}_{11}\bar G^{1}_{\bar 1\bar 1}}{4}-\frac{{\bar K}^1_{\bar 1}({\bar K}^1_{\bar 1}({\bar K}^1_{\bar 1}(G^{\bar 1}_{11})))G^{\bar 1}_{11}+{K}^{\bar 1}_1({K}^{\bar 1}_1({K}^{\bar 1}_1(\bar G^{1}_{\bar 1\bar 1})))\bar G^{1}_{\bar 1\bar 1}}{24}\\
&-\frac{\mathcal{L}_1({\bar K}^1_{\bar 1}({\bar K}^1_{\bar 1}(G^{\bar 1}_{11})))+\mathcal{\bar L}_{\bar 1}({K}^{\bar 1}_1({K}^{\bar 1}_1(\bar G^{1}_{\bar 1\bar 1})))}{24}-\frac{{K}^{\bar 1}_1(\bar G^{1}_{\bar 1\bar 1}){\bar K}^1_{\bar 1}(G^{\bar 1}_{11})}{12}\\
&-\frac{{K}^{\bar 1}_1(\bar G^{1}_{\bar 1\bar 1}){K}^{\bar 1}_1({K}^{\bar 1}_1(\bar G^{1}_{\bar 1\bar 1}))+{\bar K}^1_{\bar 1}(G^{\bar 1}_{11}){\bar K}^1_{\bar 1}({\bar K}^1_{\bar 1}(G^{\bar 1}_{11}))}{72},\\
\end{align*}
where $\bar K^1_{\bar 1}(f)=\mathcal{K}^{\bar 1}_1(f)+a(f)z_2f$ and $a(f)$ are the integers representing the action of $W(\mathcal{K}^{\bar 1}_1),V(\mathcal{K}^{\bar 1}_1)$ in $\fg_{0,0}$ on the $\fg_{0,0}$--module in which $f$ takes values. For example, $a(G^{\bar 1}_{11})=1$ or $a(\bar K^1_{\bar 1}(G^{\bar 1}_{11}))=0$.

This provides

\begin{align*}
\mathcal{W}&=-\frac{1}{3}R_{-1,-1;0,-1}^{0,-1}=-\frac{2\bar K^1_{\bar 1}(G^{\bar 1}_{11})}{9}-\frac{K^{\bar 1}_1(K^{\bar 1}_1(\bar G^{1}_{\bar 1\bar 1}))}{9}\\
&=-F\frac{2z_2}{3(\bar z_2 z_2-1)^2}+F_{1}\frac{2(z_2 \bar z_1+z_1)}{9(\bar z_2 z_2-1)^2}+F_{2}\frac{2(z_2 \bar z_1+z_1)}{9(\bar z_2 z_2-1)}\\
&+iF_{4}\frac{2(z_2 \bar z_1+z_1)^2}{9(\bar z_2 z_2-1)}-\bar F\frac{2\bar z_2^2}{3(\bar z_2 z_2-1)^2}+\bar F_{1}\frac{4\bar z_2(\bar z_2 z_1+\bar z_1)}{9(\bar z_2 z_2-1)^2}+\bar F_{2}\frac{4\bar z_2}{9(\bar z_2 z_2-1)}\\
&-i\bar F_{4}\frac{4\bar z_2(\bar z_2 z_1+\bar z_1)^2}{9(\bar z_2 z_2-1)}
-\bar F_{1,1}\frac{(\bar z_2 z_1+\bar z_1)^2}{9(\bar z_2 z_2-1)^2}
-\bar F_{1,2}\frac{2(\bar z_2 z_1+\bar z_1)}{9(\bar z_2 z_2-1)}
+i\bar F_{1,4}\frac{2(\bar z_2 z_1+\bar z_1)^3}{9(\bar z_2 z_2-1)^3}
\\&-\bar F_{2,2}\frac{1}{9}
+i\bar F_{2,4}\frac{2(\bar z_2 z_1+\bar z_1)^2}{9(\bar z_2 z_2-1)^2}
+\bar F_{4,4}\frac{(\bar z_2 z_1+\bar z_1)^4}{9(\bar z_2 z_2-1)^4}\\
\mathcal{J}&=\frac{1}{12}R_{-2,-1;-1,-1}^{0,1}=-{\frac {1}{144}}\,{ G^{\bar 1}_{11}}\,{\bar K^1_{\bar 1}(G^{\bar 1}_{11})}\,{\bar K^1_{\bar 1}(\bar K^1_{\bar 1}(G^{\bar 1}_{11}))}-\frac{1}{12}\,{
[\mathcal{L}_1,\mathcal{\bar L}_{\bar 1}](G^{\bar 1}_{11})}\\
&+{\frac {1}{324}}\,{{\bar K^1_{\bar 1}(G^{\bar 1}_{11})}}^{3}+{\frac {1}{144}}\,{
\mathcal{L}_1(\mathcal{L}_1(\bar K^1_{\bar 1}(G^{\bar 1}_{11})))}-\frac{1}{48}\,{ \mathcal{L}_1(\mathcal{\bar L}_{\bar 1}(G^{\bar 1}_{11}))}\\
&+{\frac {1}{144}}\,{
\mathcal{L}_1(G^{\bar 1}_{11})}\,{\bar K^1_{\bar 1}(\bar K^1_{\bar 1}(G^{\bar 1}_{11}))}+{\frac {1}{144}}\,{ 
\bar K^1_{\bar 1}(\bar K^1_{\bar 1}(\bar K^1_{\bar 1}(G^{\bar 1}_{11})))}\,{({ G^{\bar 1}_{11}})}^{2}\\
&+{\frac {1}{72}}\,{ G^{\bar 1}_{11}}\,{ 
\mathcal{L}_1(\bar K^1_{\bar 1}(\bar K^1_{\bar 1}(G^{\bar 1}_{11})))}+{\frac {1}{72}}\,{ \bar K^1_{\bar 1}(G^{\bar 1}_{11})}\,{\mathcal{\bar L}_{\bar 1}(G^{\bar 1}_{11})}-{
\frac {1}{72}}\,{\bar K^1_{\bar 1}(G^{\bar 1}_{11})}\,{ \mathcal{L}_1(\bar K^1_{\bar 1}(G^{\bar 1}_{11}))}\\
&-{\frac {1}{72}}\,{
 G^{\bar 1}_{11}}\,{\mathcal{\bar L}_{\bar 1}(\bar K^1_{\bar 1}(G^{\bar 1}_{11}))}.
\end{align*}
The formula for $\mathcal{J}$ in terms of the function $F$ is too long to be presented here.

\end{document}